\documentclass{article}
\usepackage[utf8]{inputenc}
\usepackage[T1]{fontenc}
\usepackage[english]{babel}
\usepackage{amsfonts,amsmath,amssymb}
\usepackage{mathrsfs}
\usepackage{enumitem}
\usepackage{stmaryrd}
\usepackage{amsthm}
\usepackage{xspace}
\usepackage{color}
\usepackage{geometry}
\usepackage{bbold}  
\usepackage{subfiles}
\usepackage{csquotes} 
\usepackage{biblatex} 

\usepackage[todonotes={textsize=scriptsize}]{changes}
\usepackage[colorlinks=true,linkcolor=blue, citecolor=red]{hyperref}
\usepackage{color}

\usepackage{microtype} 

\AtEveryBibitem{
	\clearfield{urlyear}
	\clearfield{urlmonth}
}

\newtheorem{Def}{Definition}[section]
\newtheorem{Prop}{Proposition}[section]
\newtheorem{Lemma}[Prop]{Lemma}

\newtheorem*{Rmk}{Remark}
\newtheorem{Thm}{Theorem}
\newtheorem{Cor}[Prop]{Corollary}
\newtheorem*{Hyp}{Assumptions}

\newcommand{\R}{\ensuremath{\mathbb{R} \xspace}}

\newcommand{\T}{\ensuremath{\mathbb{T} \xspace}}
\newcommand{\norm}[1]{\left\lVert #1
	\right\rVert}
\newcommand{\QT}{\ensuremath{{Q_T} \xspace}}

\newcommand{\dd}{\mathrm{d}}

\makeatletter
\def\namedlabel#1#2{\begingroup
	#2%
	\def\@currentlabel{#2}%
	\phantomsection\label{#1}\endgroup
}
\makeatother

\newcommand{\dt}{ \, \dd t  }
\newcommand{\mean}[1]{\left[ #1
	\right]_{\T^d}}
\newcommand{\C}{\mathcal{C} }
\newcommand{\Z}{\ensuremath{\mathbb{Z} \xspace}}
\newcommand{\hs}{ {H^s} }
\newcommand{\hsp}{ {H^{s+1}} }

\title{Refined blow-up criteria and global solutions for triangular cross-diffusion systems}
\author{ Alexandre Bertolino \footnote{Sorbonne Université, CNRS, Université Paris Cité, Laboratoire Jacques-Louis Lions (LJLL), Département de Mathématiques et Applications (DMA), Ecole Normale Supérieure (ENS-PSL), F-75005 Paris, France --- \href{mailto:alexandre.bertolino@sorbonne-universite.fr}{alexandre.bertolino@sorbonne-universite.fr} }}
\date{\today}

\begin{document}
	\maketitle
	\begin{abstract}
		We study the Cauchy problem associated with a class of triangular cross-diffusion systems of Shigesada--Kawasaki--Teramoto type. We develop a self-contained well-posedness theory in $\mathcal C^0 ( [0,T]; H^s(\mathbb T^d))$ based on regularity estimates for scalar Kolmogorov equations.
		The diffusion coefficient of each species depends only on species of lower index, yielding a hierarchical structure that allows for refined blow-up criteria. Finite-time singularities can occur only through the divergence of the $L^\infty(\T^d)$ norm of the solution. Assuming polynomial growth of the nonlinearities, this criterion is refined to an $ L^p$-based blow-up condition for some finite exponent $p$, yielding a substantially weaker obstruction to global existence than classical Sobolev blow-up criteria. The proof is achieved through refined tame estimates for composition in Sobolev spaces. As an application, we prove global existence of non-negative strong solutions for two-species systems with logistic-type reaction terms in dimensions $d\le 2$.
		
	\end{abstract}
	
	\section{Introduction and main results}
	\subsection{The triangular SKT system}
	
	In this article, we are concerned with the following cross-diffusion system, set on the $d$-dimensional torus $\T^d$:
	\begin{equation} \label{eq:SKTtr}
		\left\{\begin{array}{ll}
			\partial_t u_i  - \Delta[\mu_i(u_1, \dots u_{i-1}) u_i]= r_i(u_1, \dots, u_n), \quad \forall i \in\{1 ,\dots, n \} \\
			u_i(t=0, \cdot) = u_i^0,
		\end{array} \right. 
	\end{equation}
	where $n\in \mathbb N^*$ and for all $1 \le i \le n $, $u_i^0 : \T^d \rightarrow \R$, $\mu_i: \R^{i-1} \rightarrow \R_+$ and $r_i: \R^n \rightarrow \R$ are given and $u_i : Q_T = [0,T) \times \T^d \rightarrow \R$ are the unknown. To ensure a proper parabolic setting, we work under the non-degeneracy condition $\inf \mu_i  = \alpha >0$. Equations \eqref{eq:SKTtr} stem from the SKT system, introduced by Shigesada, Kawasaki, and Teramoto in \cite{shigesada_spatial_1979}:
	\begin{equation} \label{eq:SKTog}
		\left\{\begin{array}{ll}
			\partial_t u_1  - \Delta[(d_1 + a_{11}u_1 + a_{12}u_2)u_1 ]= u_1(\rho_1 - s_{11} u_1 - s_{12} u_2), \\
			\partial_t u_2  - \Delta[(d_2 + a_{21}u_1 + a_{22}u_2)u_2 ]= u_2(\rho_2 - s_{21} u_1 - s_{22} u_2), \\
			(u_1,u_2)(t=0, \cdot) = (u_1^0, u_2^0).
		\end{array} \right. 
	\end{equation}
	
	The Cauchy problem \eqref{eq:SKTtr} is a generalization of the strictly triangular version (meaning $a_{11}= a_{21} = a_{22} = 0$) of this system. The latter describes a population model in which two species, whose densities are denoted by $u_1$ and $u_2$ interact through competition for resources and mutual agitation. Both these interactions induce non-linearities, with coefficients $a_{ij}$ quantifying the strength of interspecific ($i \neq j$) or intraspecific ($i=j$) agitation and $s_{ij}$ the competition. The coefficients $ \rho_i$ are reproduction rates, and $d_i$ are the base diffusion rate when no additional agitation is induced from other individuals. The interest in this system is fueled by the repulsive behavior induced by the cross-diffusion terms, which can lead to stable segregated equilibria \cite{breden_influence_2019,inoue_coexistence-segregation_2023}. However, the study of the associated Cauchy problem only started five years after the introduction of \eqref{eq:SKTog} with \cite{kim_smooth_1984}, which began a long hike for the existence of solutions. One big step was the development in the nineties by Amann in \cite{amann_nonhomogeneous_1993} of a general theory for quasi-linear parabolic problems that contains local well-posedness theorems in Sobolev spaces $W^{1,p}, p>d $ with explosion criterion in the associated norm. Actually, Amann's theory encompasses a broader field of application, allowing for more general domains and boundary conditions. Nowadays, most proofs of the global existence of solutions to the SKT model rely on Amann's theorem and focus on proving non-explosion in $W^{1,p}$ norm. This is for example the case in \cite{lou_global_1998,choi_existence_2003,hoang_gradient_2015} or more recently \cite{guerand_global_2023} in which bounds in strong norms are proved independently on Amann theory. The application of this theory generally comes with the cost of restrictions on the system, be it its dimension \cite{lou_global_1998} or smallness of cross-diffusion coefficients \cite{choi_existence_2003}. While other methods based on entropy estimates allowed for the construction of global weak solutions \cite{chen_analysis_2006, desvillettes_entropic_2015, dietert_persisting_2024, lepoutre_entropic_2017}, the question of their uniqueness is still not well understood \cite{chen_note_2018}. In particular, the strong setting of existing weak-strong uniqueness results is too strong to be applied to Amann's solutions \cite{chen_weakstrong_2019}. Recently, efforts have been made by Gallagher and Moussa \cite{gallagher_cauchy_2025} to develop an alternative to Amann's theory on the torus. Their approach is based on Fourier analysis and use a coefficients freezing procedure to match Amann's assumption of only a mild ellipticity condition on the system, called \emph{Petrovskii’s condition}.

	In the triangular case ($a_{21}=0$), questions such as regularity and global existence are understood under various types of restrictions on the system. 
	In \cite{guerand_global_2023}, Guerand, Menegaki and Trescases consider generalizations of the triangular SKT systems and briefly present some bibliographical material referencing the different constraints on the system that can lead to global existence of regular solutions. They highlight the importance of self-diffusion in obtaining regularity and quantify the strength of self-diffusion needed to obtain global smooth solutions. In the case without any self-diffusion at all ($a_{11}=a_{22}=0$), the theory is more sparse. Pozio and Tesei \cite{pozio_global_1990} proved global smooth existence under a smallness condition on $a_{12}$ and decay assumptions on the reaction terms. The triangular structure was also exploited by Moussa in \cite{moussa_non-local_2018} to recover \eqref{eq:SKTtr} from a non-local version, as part of a derivation program from finite population models, and as a consequence obtained global existence in a weaker $L^2(\R_+ \times \T^d)$ setting, but the uniqueness of these solutions is not guaranteed. Recently, Bouton, Desvillettes and Dietert proved in \cite{bouton_global_2025} a strong global existence result in dimension $d \le 4$.

	The present work follows the same motivations as \cite{gallagher_cauchy_2025} and builds a self-contained local existence theory in the triangular case with no self-diffusion. Due to the hierarchical structure of the system \eqref{eq:SKTtr}, the proof of local existence is simpler than in \cite{gallagher_cauchy_2025} and will not require the coefficients freezing procedure. In our simpler case, refined blow-up criteria will be obtained. Our functional setting will be the same as in \cite{gallagher_cauchy_2025}. Before presenting it and stating our main results, let us introduce some notation.

	\subsection{Notations}
	
	We fix $d \in \mathbb N$ an integer. Throughout this article, we denote $\T^d := \R^d / \mathbb Z ^d$ the $d$-dimensional torus, on which for any function $u: \T^d \rightarrow \R$ the usual $L^p$ norm, mean and $L^2$ scalar product are respectively denoted $\norm{u}_{L^p} = \left(\int_{\T^d} |u|^p\right)^{1/p} $, $\mean u = \int_{\T^d} u $ and $ \langle u, v \rangle_{L^2}$. For $s \in \R$, we denote $(H^s(\T^d), \norm{\cdot}_{\hs})$ the Sobolev spaces on $\T^d$ modeled on $L^2$ and $J^s = (I- \Delta)^{\frac s 2}$ the Fourier multiplier with symbol $(1 + 4 \pi^2 |k|^2)^{\frac s 2}$ . The inhomogeneous Sobolev norms can be expressed using the operator $J^s$ and for all $u \in H^s(\T^d)$,
	$$ \norm u _{\hs} = \norm{J^s u}_{L^2}.$$
	For all $T>0$, we also define the space-time domain $\QT=[0,T) \times \T^d$ and the functional spaces $X_T^s = \mathcal C([0,T] ; H^s(\T^d))$, $Y_T^s = L^2([0,T] ; H^s(\T^d))$ and $E_T^s = X_T^s \cap Y_T^{s+1}$ and set for $u \in E_T^s$
	$$ \norm{u}_{E_T^s}^2 = \norm{u}_{X_T^s}^2+ \norm{u}_{Y_T^{s+1}}^2 = \sup_{0 \le t \le T} \norm{u(t)}_{\hs}^2 + \int_0^T \norm{u(t)}_{\hsp}^2 \,  \dd t.$$
	
	Given two linear operators $\mathcal A$ and $\mathcal B$, their commutator is denoted $[\mathcal A, \mathcal B] = \mathcal A \mathcal B - \mathcal B \mathcal A$. If $u$ is a function, the operator of multiplication by $u$ is simply denoted $u$.

	We will write $a \lesssim b $ whenever $a \le Cb$ with a constant $C>0$ depending only on the parameters of the problem. In Section \ref{sec:kolmogorov} these parameters are $s, d $ and $\inf_\QT \mu$ and in all the following sections $C$ can also depend on $\mu_i$ and $r_i$. For any $\alpha, \beta \ge 0$, we will write $a \lesssim_{\alpha, \beta} b$ if $C$ also depends increasingly on $\alpha$ and $\beta$. Sometimes, when we want to compare certain quantities with $C$, we will not use this $\lesssim$ notation and will just write $a \le C(\alpha, \beta) b $. 
	
	Given a real number $x$, his positive part and upper integer part will be denoted respectively $x^+= \max(0,x)$ and $\lceil x \rceil = \min \{n \in \mathbb Z \mid n \ge x\}$.
	
	Distinction between vectorial and scalar quantities will be made by using uppercase for vectors and 
	lowercase for scalars. For example, the solution of \eqref{eq:SKTtr} is denoted $U=(u_1, \dots, u_n)$.  Moreover if all coordinates of a vector $U\in \R^n$ are non-negative, we write $U \ge 0$ and say that $U$ is non-negative.
	The functional spaces of vector-valued functions will be written the same way as the spaces of scalar functions. For example, we write $H^s(\T^d)$ instead of $H^s(\T^d)^n$. Lastly, if $f: \R^k \rightarrow \R$ with $k<n$, we still denote $f(U)$ the quantity $f(u_1, \dots, u_k)$. 
	
	\subsection{Outline and main results}	
	
	Given the structure of \eqref{eq:SKTtr}, it will be useful to study the following problem:
	\begin{Def}
		For any bounded function $\mu\in L^\infty(Q_T)$, $f\in L^1(0,T;\mathscr{D}'(\T^d))$ and any $z_0\in L^1(\T^d)$, we say that $z\in L^1(Q_T)$ is a solution of the Cauchy problem
		\begin{equation} \label{eq:Kolmogorov}
			\left\{\begin{array}{ll}
				\partial_t z  - \Delta[\mu z]= f \\
				z(0, \cdot) = z^0 
			\end{array} \right. 
		\end{equation}
		if, for all $\varphi\in\mathscr{C}_c^\infty([0,T)\times\T^d)$  there holds
		\begin{align*}
			-\int_{Q_T} z (\partial_t \varphi + \mu \Delta\varphi) = \int_{\T^d}z^0 \varphi(0) + \int_0^T \langle f(t),\varphi(t)\rangle\,\dd t.
		\end{align*}
		
	\end{Def}
	This frame produces a notion of solution for \eqref{eq:SKTtr} which we see from now on as a system of $n$ Cauchy problems of type \eqref{eq:Kolmogorov}. Equation \eqref{eq:Kolmogorov} is called the Kolmogorov equation and its study in Section \ref{sec:kolmogorov} will allow to develop estimates on solutions of systems of type \eqref{eq:SKTtr}. To do so, we will rely on higher regularity variants of the duality lemma used in \cite{bansaye_stability_2025}. Before stating our main result, let us introduce our assumptions. 
	\begin{Hyp} 
		We consider the following set of assumptions: \begin{enumerate}[label=(\roman*)]
			\item[\namedlabel{ass:1}{(A1)}] $d \ge 1$ and $s > d /2  $.
			\item[\namedlabel{ass:2}{(A2)}] For all $k \in \{1, \dots, n\}$, the functions $r_k: \R^n \rightarrow \R$ and $\mu_k : \R^{k-1} \rightarrow \R $ are smooth.
			\item[\namedlabel{ass:3}{(A3)}] For all $k \in \{1, \dots, n\}$, $\inf_{\R ^{k-1}} \mu_k > 0$.
			\item[\namedlabel{ass:4}{(A4)}] For all $k \in \{1, \dots, n\}$, $r_k$ has the form $r_k(X) = x_k \bar r_k (X)$ with $\bar r_k \in \mathcal C^\infty(\R^n;\R)$
			\item[\namedlabel{ass:5}{(A5)}] For all $k \in \{1, \dots, n\}$,	there exists constants $C>0$ and $\alpha\ge 1$ such that $|\nabla^{\lceil \frac d 2 \rceil +2} \mu_k(X)|$, $|\nabla^{\lceil \frac d 2 \rceil} r_k(X)| \le C(1 + |X|^\alpha)$ for all $X \in R^n$. 
			\item[\namedlabel{ass:6}{(A6)}] Under \ref{ass:4}, there exists constants $C>0$ and $\alpha \ge 1$ such that for all $k \in \{1, \dots, n\}$, $ -C(1 + |X|) \lesssim \bar r_k(X)\lesssim C $ and $|\nabla \bar r_k (X)| \le C(1 + |X|^\alpha)$ for all $X \in R^n$.
		\end{enumerate}
		
	\end{Hyp}
	Conditions \ref{ass:1} and \ref{ass:2} are here to ensure that composition by the $\mu_k$ and $r_k$ behave nicely in $H^s(\T^d)$. \ref{ass:3} is similar to a parabolicity condition and ensures that energy-type estimates are possible. Assumption \ref{ass:4} is typical of poppulation dynamics systems and means that there is a notion of birth and death rate by individual. Assumption \ref{ass:5} is satisfied whenever the $ \mu_k$ and $r_k$ are polynomials. Assumption \ref{ass:6} contains a weaker version of \ref{ass:5} and a bound that ensures that the total population cannot blow up in finite time.

	The main result is the following:
	\begin{Thm} \label{th:localExist} Consider assumptions \ref{ass:1}–\ref{ass:3} and let $U^0 \in H^s(\T^d)$. There exists $T>0$ and a unique solution $U \in E_T^s$ of \eqref{eq:SKTtr}. Moreover, if $T^*$ is the maximal lifetime of $U$, then $U \in \mathcal C^\infty( (0, T^*) \times \T^d)$ and if $T^* < \infty$,
		$$  \norm{U(t)}_{L^\infty(\T^d)} \underset{t \to T^*}{\longrightarrow} \infty .$$  
		Additionally, if $V^0 \in H^s(\T^d)$ and the Cauchy problem \eqref{eq:SKTtr} admits a solution $V \in E_T^s$, the following estimate holds:
		$$ \norm{U- V}_{E_T^s} \lesssim_{\norm{U}_{E_T^s} ,\norm{V}_{E_T^s} } \norm{U^0- V^0}_{\hs} .  $$
		Lastly, considering the additional assumption \ref{ass:4}, if all components of $U_0$ are non-negative on $\T^d$ then so are the components of $U$ on $Q_{T^*}$.
	\end{Thm}
	
	This theorem is proved by a fixed point procedure which will take roots in estimates for a version of \eqref{eq:SKTtr} where the $r_k(U)$ are replaced by generic source terms. It is to be compared with Theorem 1 of \cite{gallagher_cauchy_2025}, which also states local well-posedness. Theorem 2 of \cite{gallagher_cauchy_2025} states a blow-up criterion in $H^s(\T^d)$ and that under a smallness condition on the initial data, the solutions are global. Due to the triangular structure of our system, we can prove a finer blow-up criterion in $L^\infty(\T^d)$. Under an additional assumption of polynomial growth of the parameters $\mu_k, r_k$, this criterion may be refined in the following way.
	
	\begin{Prop} \label{prop:exploLp}
		Consider \ref{ass:1}–\ref{ass:3} and \ref{ass:5}. There exists $p < \infty$ such that for all $ U^0 \in H^s (\T^d)$, the maximal lifetime $T^*$ of the solution $U \in E_T^s $ of \eqref{eq:SKTtr} satisfies either $T^* = + \infty $ or
		$$ \int_0^T \norm{U(t)}_{L^p(\T^d)}^p  \dt \underset{T \to T^*}{\longrightarrow} \infty .$$
	\end{Prop} 
	Hints towards such an explosion criterion were already given by Choi, Lui and Yamada in \cite{choi_existence_2003} in the case $n=2$, but their argument requires both a smallness assumption on the cross-diffusion parameter. Still in the case $n=2$, Bouton, Desvillettes and Dietert proved in Section 3 of \cite{bouton_global_2025} that if explosion happens in $L^p(\T^d)$ for some $p\ge 1$, then it happens for all $p > 1 + \frac d 2$. The proof of Proposition \ref{prop:exploLp} uses refined tame estimates for the composition in Sobolev spaces, which we state and prove in Appendix \ref{sec:compo}.
	Finally, as an application of the two previous results, Section \ref{sec:5} proves a global existence result in dimension $d\le2$.
	
	\begin{Thm}\label{th:global} 
		Consider \ref{ass:1}–\ref{ass:4} and \ref{ass:6} and assume $d \le 2$, $n =2$ and $U^0 \in H^s(\T^d)$. If $U^0$ is non-negative, then the solution $U$ of problem \eqref{eq:SKTtr} is global (\textit{i.e.}, $T^* = \infty$).
	\end{Thm}
	
	Note that under \ref{ass:4}, since the solutions of \eqref{eq:SKTtr} stay non-negative, when $U^0$ is non-negative the assumptions on $R$ and the $\mu_k$ can be weakened to assumptions on the restricted functions $R_{|\R^n_+}$ and $\mu_{k|\R^{k-1}_+}$. This observation implies that Theorem \ref{th:global} applies to the logistic case, which consists in considering \eqref{eq:SKTog} with $a_{11}=a_{22}= a_{21}=0$. In this case, the result was already proved in \cite{lou_global_1998} for $d\le 2$ and more recently by \cite{bouton_global_2025} for $d\le 4$. In these two approaches, global existence follows from determining bounds independently on Amann's theory and then applying Amann's blow-up criterion. Our blow-up criterion in Proposition \ref{prop:exploLp} happens in a weaker norm than Amann's one and we therefore hope that it will simplify further attemps at proving global existence results.\\

	\emph{Outline of this article.} We first study the Kolmogorov equation \ref{eq:Kolmogorov} in Section \ref{sec:kolmogorov}, where we start from $E_T^{-1}$ estimates to produce $E_T^s$ ones for $s>d/2$. Section \ref{sec:3} is dedicated to the proof of Theorem \ref{th:localExist} and Proposition \ref{prop:exploLp}. The proof of Proposition \ref{prop:exploLp} uses a new tame composition estimate that we prove in Appendix \ref{sec:compo}. Section \ref{sec:5} is a proof of Theorem \ref{th:global}. For the sake of completeness, we included in Appendix \ref{sec:transfer} a transference principle that will allow us to use commutator estimates previously proved on $\R^d$.

	\subsection{Acknowledgment } The author warmly thanks Ayman Moussa for his advice throughout this project and for thoroughly reading the drafts. This work was partially funded by the Chair “Modélisation Mathématique et Biodiversité" of VEOLIA-Ecole Polytechnique-MNHN-F.X. and by the European Union (ERC, SINGER, 101054787). Views and opinions expressed are however those of the author(s) only and do not necessarily reflect those of the European Union or the European Research Council. Neither the European Union nor the granting authority can be held responsible for them.
	
	\section{The Kolmogorov equation} \label{sec:kolmogorov}
	
	In this section, we fix $T>0$ a time horizon and investigate equation \eqref{eq:Kolmogorov} on $\QT$. Since the work of Moussa in \cite{moussa_non-local_2018}, existence of solutions and stability estimates are already known in an $L^2(\QT)$ setting. 
	The lemma we propose to prove is to assimilate with the more recent duality lemma presented in \cite{bansaye_stability_2025}, for which we give two variants. The first one happens in the $E_T^{-1}$ setting for $z$ and allows to eliminate the negative part of the source term, which will be useful to treat the reaction terms of \eqref{eq:SKTtr} when they are linearly upper-bounded. From this first estimate we will then derive the second one in the $E_T^{s}$ setting for $s > d/2$.

	\subsection{Energy estimate in inhomogeneous norm}
	
	The energy estimates for \eqref{eq:Kolmogorov} presented in \cite{bansaye_stability_2025} bound $z$ in homogeneous Sobolev norm but we chose to write our next estimate in inhomogeneous norms because we want to use the sign-preserving property of the inhomogeneous operator $J^{-2}$. 
	
	\begin{Lemma} \label{lem:duality-1}
		Consider $\mu \in L^\infty(\QT)$ such that $\alpha := \inf_{\QT} \mu > 0$, $z^0 \in H^{-1}(\T^d)$ and $f \in L^2_TH^{-2}$. Then there exists a unique $z \in L^2(\QT)$ that weakly solves the Kolmogorov equation \eqref{eq:Kolmogorov}. 
		Furthermore, this solution $z$ belongs to $\C([0,T];H^{-1}(\T))$ and there exists a constant $C >0$ such that 
		\begin{equation}\label{eq:duality_estimate}
			\norm{z(T)}_{H^{-1}}^2 + \int_{\QT} \mu z^2  \lesssim  \norm{z^0}_{H^{-1}}^2 + \int_0^T \mean{\mu(t)} \norm{z(t)}_{H^{-1}}^2 \, \dd t + \int_0^T \norm{f(t)}_{H^{-2}}^2 \, \dd t .
		\end{equation}
		Moreover, if $z$ is non-negative, then $f$ can be replaced by any function $\tilde f$ such that $f \le \tilde f$ in the above estimate.
	\end{Lemma} 
	\begin{proof}
		Lemma 2.1 from \cite{bansaye_stability_2025} asserts that global existence, uniqueness and $\C([0,T];H^{-1}(\T))$ regularity for \eqref{eq:Kolmogorov} hold if $f$ can be decomposed as $f = \Delta g + r$ with $g, r \in L^2(\QT)$, which is possible since $f = \Delta \left( \Delta^{-1}(f - \mean f)\right) + \mean f$ and $\Delta^{-1}$ defines a continuous operator from the space of mean-free $H^{-2}(\T^d)$ distributions to the space of mean-free $L^2(\T^d)$ functions. We now focus on the duality estimate \eqref{eq:duality_estimate} which needs to be proven only in a smooth setting. \\

		For all $k \in \mathbb{N}$, set $\Phi_k = J^{-2k} z = (I - \Delta)^{-k} z$ and reformulate \eqref{eq:Kolmogorov} as 
		\begin{equation} \label{eq:ineqKolmoMod}
			\partial_t z  + \mu z - \Delta( \mu z) = \mu z + f.
		\end{equation}
		If $k \ge 1$, we multiply this equation by $\Phi_k$ and integrate over $\T^d$ using the fact that $J$ is a symmetric operator to see that for all $t \in [0,T]$,
		\begin{equation} \label{1652crepe}
			\langle J^{-2k} z(t), \partial_ t z(t) \rangle + \langle (I - \Delta)^{-k+1} z(t) , \mu(t) z(t) \rangle = \langle \Phi_k(t), \mu(t) z(t) \rangle + \langle \Phi_k(t) , f(t) \rangle.
		\end{equation}
		Here, if $z$ is non-negative, applying the maximum principle (see \cite[Theorem 6.4.2]{evans_partial_1988} for a detailed discussion of maximum principles) $k$ times successively to $\Phi_1, \dots, \Phi_k$, we see that $\Phi_k$ is also non-negative and therefore $\langle \Phi_k(t) , f(t) \rangle \le \langle \Phi_k(t) , \tilde f (t) \rangle$. The rest of the proof in the case when $z \ge 0$ is the same as the general proof but with $f$ replaced by $\tilde f $. By Young's inequality,
		$$ \langle \Phi_k(t) , f(t) \rangle = \langle \Phi_{k-1}(t) , J^{-2} f \rangle  \le \frac \alpha 2 \norm{\Phi_{k-1}(t)}_{L^2}^2 + \frac 1 {2\alpha} \norm{J^{-2}f(t)}_{L^2}^2 = \frac \alpha 2 \norm{z(t)}_{H^{-2k + 2}}^2 + \frac 1 {2\alpha} \norm{f(t)}_{H^{-2}}^2  .$$
		Using $ \norm{\cdot}_{H^{-2k+2}} \le \norm{\cdot}_{H^{-k+1}}$ and plugging this into \eqref{1652crepe},
		$$\frac 1 2\frac \dd {\dd t} \norm{z(t)}_{H^{-k}} + \langle \Phi_{k-1}(t) z(t) , \mu(t) z(t) \rangle \le \langle \Phi_k(t), \mu(t) z(t) \rangle +\frac \alpha 2 \norm{z(t)}_{H^{-k+1}}^2 + \frac 1 {2\alpha} \norm{f(t)}_{H^{-2}}^2. $$
		Let $k_0 \in \mathbb N$ be large enough to have $H^{2k_0 -1}(\T^d) \hookrightarrow L^\infty(\T^d)$ (\textit{i.e.}, $k_0 > \frac d 4 + \frac 1 2$). Then, summing the previous inequality over $k= 1, \dots, k_0$ we get
		$$\frac 1 2 \sum_{k=1}^{k_0}\frac \dd {\dd t} \norm{z(t)}_{H^{-k}} + \sum_{k=0}^{k_0-1}\langle \Phi_{k}(t) z(t) , \mu(t) z(t) \rangle \le \sum_{k=1}^{k_0} \langle \Phi_k(t), \mu(t) z(t) \rangle +\frac \alpha 2 \sum_{k=0}^{k_0-1} \norm{z(t)}_{H^{-k}}^2 + \frac {k_0} {2\alpha} \norm{f(t)}_{H^{-2}}^2. $$
		Denoting $\psi(t) = \sum_{k=1}^{k_0} \norm{z(t)}_{H^{-k}}^2$ and simplifying the terms appearing on both sides,
		$$ \frac 1 2 \frac \dd {\dd t} \psi(t) +\int_{\T^d} \mu(t) z(t)^2  \le \int_{\T^d} \Phi_{k_0}(t) \mu(t) z(t) +\frac \alpha 2 \int_{\T^d} z(t)^2 + \frac \alpha 2 \psi(t) + \frac {k_0} {2\alpha} \norm{f(t)}_{H^{-2}}^2. $$
		Using once again Young's inequality,
		$$\int_{\T^d} \Phi_{k_0}(t) \mu(t) z(t) \lesssim 2 \int_{\T^d} \Phi_{k_0}(t)^2 \mu(t) \frac 1 4 \int_{\T^d} \mu(t) z(t)^2,$$
		and since $\mu \ge \alpha$, $\int_{\T^d} \left(\frac \alpha 2 + \frac {\mu(t)} 4 \right) z(t)^2$ can be absorbed in the left-hand side to have 
		$$ \frac 1 2 \frac \dd {\dd t} \psi(t) +\frac 1 4 \int_{\T^d} \mu(t) z(t)^2  \le 2 \norm{\Phi_{k_0}(t)}_{L^\infty(\T^d)}^2 \int_{\T^d} \mu(t)+ \frac \alpha 2 \psi(t) + \frac {k_0} {2\alpha} \norm{f(t)}_{H^{-2}}^2. $$
		By Sobolev embedding, we have 
		$$\norm{\Phi_{k_0}(t)}_{L^\infty(\T^d)} \lesssim \norm{\Phi_{k_0}(t)}_{2k_0 -1} = \norm{J ^{-2k_0}z(t)}_{2k_0 -1} =\norm{z(t)}_{H^{-1}} \lesssim \psi(t),  $$
		
		so, since $\mu \ge \alpha$, integrating in time we get
		$$ \psi(T) + \int_{\QT} \mu z^2  \lesssim  \psi(0) \int_0^T \mean{\mu(t)} \psi(t) \, \dd t + \int_0^T \norm{f(t)}_{H^{-2}}^2 \, \dd t . $$
		Since $\psi(t) =\sum_{k=1}^{k_0} \norm{z(t)}_{H^{-k}}^2 \le k_0 \norm{z(t)}_{H^{-1}}^2 $, this is \eqref{eq:duality_estimate}.
		
	\end{proof}

	\subsection{ $H^s$ estimate}
	
	We prove that similar estimates hold in a higher regularity setting. The next general estimate will be the source of multiple high regularity variants of \eqref{eq:duality_estimate}.
	\begin{Prop} \label{prop:dualiteHs}
		Let $s \ge -1$,  $z^0 \in H^s(\T^d), f \in Y_T^{s-1}$ and $ \mu \in  L^\infty(\QT)$ such that $\inf_\QT \mu >0$.
		Then for all $z \in E_T^s$ satisfying \eqref{eq:Kolmogorov} and $[\mu, J^{s+1}] z \in L ^2(\QT)$, the following estimate holds:
		\begin{equation} \label{eq:dualityHsOrigin}
			\norm{z}_{E_T^s}^2 \lesssim   \norm{z^0}_{\hs}^2+ \int_0^T \mean{\mu(t)} \norm{z(t)}_{\hs}^2\, \dd t + \norm{f}_{Y_T^{s-1}}^2 + \int_0^T \norm{[\mu(t), J^{s+1}] z(t)}_{L^2}^2\, \dd t.
		\end{equation}
	\end{Prop}
	
	\begin{proof}
		
		We apply $J^{s+1}$ to \eqref{eq:Kolmogorov} and get
		$$ \partial_t J^{s+1} z - \Delta (\mu J^{s+1} z) = J^{s+1} f + \Delta([\mu, J^{s+1}] z ) .$$
		By the duality estimate \eqref{eq:duality_estimate}
		\begin{multline*}
			\norm{J^{s+1} z}_{L^\infty_T H^{-1}}^2 + \norm{J^{s+1} z}_{L^2(Q_T)}^2 \lesssim  \norm{J^{s+1} z^0}_{H^{-1}}^2 
			+\int_0^T \mean{\mu(t)} \norm{J^{s+1}z(t)}_{H^{-1}}^2\, \dd t\\
			+\int_0^T \norm{J^{s+1}f(t)}_{H^{-2}}^2\, \dd t 
			+ \int_0^T \norm{\Delta([\mu(t), J^{s+1}] z(t) )}_{H^{-2}}^2\, \dd t .
		\end{multline*} 
		Since $\norm{J^{s+1} z}_{H^{-1}}=\norm{z}_{\hs}$, $\norm{J^{s+1} f}_{H^{-2}}=\norm{f}_{H^{s-1}}$ and $\Delta$ is a continuous operator from $L^2(\T^d)$ to $H^{-2}(\T^d)$, this is the estimate in the proposition.
	\end{proof}

	Depending on how the commutator $[\mu, J^{s+1}]z$ is bounded, one gets various estimates:

	\begin{Cor} \label{cor:duaHsFinal}
		Let $s > d/2$,  $z^0 \in H^s(\T^d), f \in Y_T^{s-1}$ and $ \mu \in Y_T^{s+1}\cap L^\infty(\QT)$ such that $ \inf_\QT \mu >0$.
		Then there exists a constant $c>0$ such that for all $z \in E_T^s$ satisfying \eqref{eq:Kolmogorov}, the following estimate holds:
		\begin{equation*}
			\norm{z}_{E_T^s}^2 \lesssim  \exp(cT + c \norm{\mu}_{Y_T^{s+1}}^2 )\left( \norm{z^0}_{\hs}^2 + \norm{f}_{Y_T^{s-1}}^2 \right).
		\end{equation*}
	\end{Cor}
\begin{proof}
	To prove this first bound, we rely on the commutator estimate in Corollary \ref{cor:kato} in the appendix, which states that
	$$ \norm{[\mu(t), J^{s+1}] z(t)}_{L^2(\T^d)} \lesssim \norm{\nabla \mu(t)}_{\hs}\norm{z(t)}_{\hs} \lesssim \norm{\mu(t)}_{\hsp}\norm{z(t)}_{\hs}.$$
	Injecting this in \eqref{eq:dualityHsOrigin} and bounding $\mean{\mu(t)} \le \norm{\mu(t)}_{\hsp} $ reveals
	$$ \norm{z}_{E_T^s}^2 \lesssim   \norm{z^0}_{\hs}^2+  \norm{f}_{Y_T^{s-1}}^2 + \int_0^T (\norm{\mu(t)}_{\hsp} + \norm{\mu(t)}_{\hsp}^2) \norm{z(t)}_{\hs}^2\, \dd t,$$
	so the result follows from Young's inequality $\norm{\mu(t)}_{\hsp} \lesssim 1 + \norm{\mu(t)}_{\hsp}^2$ and Grönwall's lemma.
	
\end{proof}

As a consequence of this first bound we get a way to compare solutions to Kolmogorov equations with different data.
\begin{Cor} \label{cor:dualStab}
	Let $s> d/2$, $u^0, v^0 \in H^s(\T^d)$, $f,f' \in Y_T^{s-1}$ and $ \mu, \mu'\in Y_T^{s+1}\cap L^\infty(\QT)$ such that $\inf_\QT  \mu > 0 $. Let $u$ and $v$ be solutions of \eqref{eq:Kolmogorov} with respective data $(u^0, f , \mu)$ and $(v^0, f' , \mu')$. Then, for some constant $c$, the following estimate holds:
	$$\norm{u-v}_{E_T^s}^2 \lesssim  \exp(cT + c \norm{\mu}_{Y_T^{s+1}}^2 )\left( \norm{u^0 - v^0}_{\hs}^2 + \norm{f- f'}_{Y_T^{s-1}}^2 + \norm{\mu -\mu' }_{E_T^{s}}^2\norm{v }_{E_T^{s}}^2 \right).$$
	Both $c$ and the constant behind $\lesssim$ do not depend on $\inf_\QT  \mu'$.
\end{Cor}
\begin{proof}
	Let $z = u-v$. It satisfies
	$$ \partial_t z - \Delta(\mu z ) = f - f' + \Delta((\mu -\mu') v).$$
	By the previous corollary,
	$$\norm{z}_{E_T^s}^2 \lesssim  \exp(cT + c \norm{\mu}_{Y_T^{s+1}}^2 )\left( \norm{z^0}_{\hs}^2 + \norm{f- f' +\Delta((\mu -\mu') v) }_{Y_T^{s-1}}^2 \right).$$
	Using the tame product estimate for Sobolev spaces, for all $t \in [0,T]$,
	\begin{align*}
		\norm{\Delta((\mu(t) -\mu'(t)) v(t))}_{H^{s-1}} & \lesssim \norm{(\mu(t) -\mu'(t)) v(t)}_{\hsp},\\
		&\lesssim \norm{\mu(t) -\mu'(t)}_{L^\infty}\norm{v(t)}_{\hsp} + \norm{\mu(t) -\mu'(t)}_{\hsp}\norm{ v(t)}_{L^\infty}. 
	\end{align*}
	Using the Sobolev embedding $\norm{ \cdot }_{L^\infty} \lesssim \norm{ \cdot }_{\hs}$ and integrating in time,
	$$\norm{\Delta((\mu -\mu') v) }_{Y_T^{s-1}}^2 \lesssim \norm{\mu -\mu' }_{X_T^{s}}^2\norm{v }_{Y_T^{s+1}}^2 + \norm{\mu -\mu' }_{Y_T^{s+1}}^2\norm{v }_{X_T^{s}}^2. $$
	Injecting this in the bound from the beginning of the proof yields
	\begin{multline*}
		\norm{u-v}_{E_T^s}^2 \lesssim  \exp(cT + c \norm{\mu}_{Y_T^{s+1}}^2 )\big( \norm{u^0 - v^0}_{\hs}^2 + \norm{f- f'}_{Y_T^{s-1}}^2 \\
		+ \norm{\mu -\mu' }_{X_T^{s}}^2\norm{v }_{Y_T^{s+1}}^2 + \norm{\mu -\mu' }_{Y_T^{s+1}}^2\norm{v }_{X_T^{s}}^2\big),
	\end{multline*}
	hence the result since by definition $\norm{\cdot}_{X_T^s},\norm{\cdot}_{Y_T^{s+1}} \le \norm{\cdot}_{E_T^s}$
\end{proof}

The next estimate will be useful for proving explosion criteria:

\begin{Cor} \label{dualiteHsalt}
	Let $s > d/2$ and $\varepsilon \ge 0$. Then there exists $p  \in (2, \infty] $ and an increasing function $P: \R^2 \rightarrow \R^+$ such that for all $z^0 \in H^s(\T^d), f \in Y_T^{s-1},  \mu \in Y_T^{s+1 + \varepsilon}\cap L^\infty(\QT)$ such that $\inf_\QT \mu >0$ and $z \in E_T^s$ satisfying \eqref{eq:Kolmogorov}, the following estimate holds:
	\begin{equation*}
		\norm{z}_{E_T^s}^2 \lesssim  \norm{z^0}_{\hs}^2+ \int_0^T \mean{\mu(t)} \norm{z(t)}_{\hs}^2\, \dd t + \norm{f}_{Y_T^{s-1}}^2 
		+ \int_0^T P(\norm{ \mu(t)}_{L^p}, \norm{ z(t)}_{L^p})(1+\norm{\mu(t)}_{H^{s+1 +\varepsilon}}^2) \, \dd t.
	\end{equation*}
	Moreover, if $\varepsilon >0$, $p$ can be chosen finite.
\end{Cor}
\begin{proof}
	We start from \eqref{eq:dualityHsOrigin} and apply the commutator estimate in Lemma \ref{lem:comutLp} in the appendix, which states that there exists $p \ge 2$ which is finite if $\varepsilon >0$ such that for all $\delta >0$, there exists an increasing function $P_\delta: \R^2 \rightarrow \R_+$ which is polynomial if $\varepsilon >0$, satisfying
	$$ \norm{[J^{s+1}, \mu((t)]z(t)}_{L^{2}} \lesssim P_\delta(\norm{ \mu(t)}_{L^p}, \norm{ z(t)}_{L^p})(1+\norm{\mu(t)}_{H ^{s+1 +\varepsilon}}) + \delta \norm{ z(t)}_{H^{s+1}}.$$
	
	Injecting this estimate in \eqref{eq:dualityHsOrigin}, we get
	\begin{multline*}
		\norm{z}_{E_T^s}^2 \lesssim  \norm{z^0}_{\hs}^2+ \int_0^T \mean{\mu(t)} \norm{z(t)}_{\hs}^2\, \dd t + \norm{f}_{Y_T^{s-1}}^2 \\
		+ \int_0^T \left(P_\delta(\norm{ \mu(t)}_{L^p}, \norm{ z(t)}_{L^p})(1+\norm{\mu(t)}_{H^{s+1 +\varepsilon}}^2) + \delta^2 \norm{ z(t)}_{H^{s+1}}^2 \right)\, \dd t.
	\end{multline*} 
	Fixing $\delta $ to be small enough to absorb the of $\norm{ z(t)}_{H^{s+1}}^2$ in the left-hand side, we conclude.
	
\end{proof}

\subsection{Regular solutions to the Kolmogorov equation}

We now prove that in a regular enough setting for $\mu$, the weak solutions produced by Lemma \ref{lem:duality-1} are regular enough to apply the estimates we just discussed.
\begin{Prop} \label{solveKolmo}
	Let $s> d/2$, $f \in Y_T^{s-1}$, $\mu \in E_T^s$ and $z^0 \in H^s(\T^d)$. Then the unique solution $z \in L^2(\QT)$ of \eqref{eq:Kolmogorov} is in $E_T^s $
\end{Prop}
\begin{proof}
	For smooth data, the existence is a consequence of classical parabolic theory. We will deduce the wanted result by approximating the solution and data by smooth functions. By convolution with a space-time mollification kernel, there exist sequences $(\mu_N)$, $(f_N)$ and $(z^0_N)$ of smooth functions that converge respectively in $ E_T^{s},  Y_T^{s-1}$ and $H^{s}(\T^d)$ to $\mu, f $ and $z^0$. Let $(z_N)$ be the associated sequence of smooth solutions. We will prove that $(z_N)$ is a Cauchy sequence in $E_T^s$. By uniqueness of solutions of \eqref{eq:Kolmogorov} in $L^2(\QT)$, $z$ will then be the only possible limit of $(z_N)$.

	Since $\norm{z^0_N}_{\hs} \lesssim \norm{z^0}_{\hs}$, $\norm{f_N}_{Y_T^{s-1}} \lesssim \norm{f}_{Y_T^{s-1}} $ and $\norm{\mu_N}_{E_T^s} \lesssim \norm{\mu}_{E_T^s}$, Corollary \ref{cor:duaHsFinal} proves that $(z_N)$ is a bounded sequence in $E_T^s$:
	$$ \norm{z_N}_{E_T^s}^2 \lesssim  \exp(cT + c \norm{\mu_N}_{Y_T^{s+1}}^2 )\left( \norm{z^0_N}_{\hs}^2 + \norm{f_N}_{Y_T^{s-1}}^2 \right)\lesssim_{T,\norm{z^0}_{\hs}, \norm{f}_{Y_T^{s-1}},\norm{\mu}_{E_T^s}} \, 1. $$
	By Corollary \ref{cor:dualStab}, for all $M,N \in \mathbb N$,
	\begin{align*}
		\norm{z_N-z_M}_{E_T^s}^2 &\lesssim  \exp(cT + c \norm{\mu_N}_{Y_T^{s+1}}^2 )\left( \norm{z^0_N - z^0_M}_{\hs}^2 + \norm{f_N- f_M}_{Y_T^{s-1}}^2 + \norm{\mu_N -\mu_M }_{E_T^{s}}^2\norm{z_M }_{E_T^{s}}^2 \right) ,\\
		&\lesssim_{T,\norm{z^0}_{\hs}, \norm{f}_{Y_T^{s-1}},\norm{\mu}_{E_T^s}}  \, \norm{z^0_N - z^0_M}_{\hs}^2 + \norm{f_N- f_M}_{Y_T^{s-1}}^2 + \norm{\mu_N -\mu_M }_{E_T^{s}}^2,
	\end{align*}
	so $\norm{z_N-z_M}_{E_T^s}^2$ converges to $0$ as $M,N  \to \infty$. $(z_N)$ is therefore a Cauchy sequence; this concludes the proof.
\end{proof}

\section{ Proof of Theorem \ref{th:localExist}} \label{sec:3}

In this section, we work under assumptions \ref{ass:1}–\ref{ass:3}. In the conservative case (\textit{i.e.}, $r_k=0$ for all $k$), \eqref{eq:SKTtr} can be considered as a sequence of linear parabolic equations since once $u_{k-1}$ has been determined, the equation on $u_k$ can be seen as a linear equation parameterized by $u_{k-1}$. Therefore, by induction, classical parabolic theory yields global existence for regular enough initial data. Thus, the difficulty in proving Theorem \ref{th:localExist} comes from having reaction in addition to the cross-diffusion. In a first time, we abstract away the reaction terms and get estimates for the system with a general source term in Subsection \ref{sec:triSource}. These estimates are used in the fixed point procedure in Subsection \ref{sec:fixedPoint}. The fixed point argument already gives an explosion criterion in $H^s(\T^d)$ norm but Subsection \ref{sec:explo} refines it to an $L^\infty(\T^d)$ and under the additional assumption \ref{ass:5} to an $L^p(\QT)$ explosion criterion (Proposition \ref{prop:exploLp}) for some $p< \infty$. Subsection \ref{sec:end} then concludes the proof by discussing smoothness and sign-preservation.

\subsection{ Stability estimate} \label{sec:triSource}

In this subsection we are interested in the triangular SKT system with source terms:
\begin{equation} \label{eq:SKTtriAffine}
	\left\{\begin{array}{ll}
		\partial_t u_k  - \Delta[\mu_k(u_1, \dots u_{k-1}) u_k]= f_k \quad \forall k \in\{1 ,\dots, n \} \\  
		u_k(t=0, \cdot) = u_k^0
	\end{array} \right. 
\end{equation}

\begin{Prop} \label{prop:triBound}
	Assume that
	$F = (f_1, \dots, f_n) \in Y_T^{s-1} $ and 
	$U^0 = (u_1^0, \dots , u_n^0) \in H^s(\T^d)$. Then, there exists a unique solution $U= (u_1, \dots, u_n)$ of \eqref{eq:SKTtriAffine} in $L^2(\QT)$. moreover $U \in E_T^s$ and 
	$$
	\norm{u_k}_{E_T^s}^2 
	\lesssim  \exp \left(cT + c\int_0^T \norm{\mu_k(u_1, \dots, u_{k-1})}_{\hsp}^2(t) \, \dd t \right)\left( \norm{u^0_k}_{\hs}^2 +  \norm{f_k}_{Y_T^{s-1}}^2 \right),
	$$ 
	where the constant behind $\lesssim$ only depends on the parameters $(\mu_i)$.
\end{Prop}
\begin{proof}
	The proof is by induction on $k$. The equation on $u_1$ is the regular heat equation and if existence, uniqueness and regularity are known for $u_1, \dots, u_{k-1}$, it follows from Proposition \ref{solveKolmo} and Corollary \ref{cor:duaHsFinal} that these properties are also satisfied by $u_k$, with the above estimate. It is indeed possible to apply these results since \ref{prop:bahouri} tells that if $u_1, \dots, u_{k-1} \in E_T^{s}$, then $\mu_k(U) \in E_T^s \supset Y_T^{s+1} \cap L^\infty(\QT)$.
\end{proof}

\begin{Cor} \label{cor:triBound}
	For all $F \in Y_T^{s-1} $ and
	$U^0 \in H^s(T^d)$, the unique solution $U \in L^2(\QT)$ of \eqref{eq:SKTtriAffine} satisfies
	$$ \norm{U}_{E_T^s} \le  C(T, \norm{U^0}_{\hs},\norm{F}_{Y_T^{s-1}}), $$
	where $C$ is an increasing function of its arguments, depending only on $s$ and the parameters $\mu_i$ of the problem.
\end{Cor}
\begin{proof}
	The proof is by induction. Let $k \in \{1 , \dots, n\}$. Looking at the estimate in Proposition \ref{prop:triBound}, to prove the bound on $ \norm{u_k}_{E_T^s} $, it is sufficient to bound the following integral
	\begin{align*}
		\int_0^T \norm{\mu_k(u_1, \dots, u_{k-1})}_{\hsp}^2(t) \, \dd t &\lesssim T + \int_0^T \norm{\mu_k(u_1, \dots, u_{k-1}) - \mu_k(0, \dots, 0)}_{\hsp}^2(t) \, \dd t \\
		& \lesssim_{\norm{(u_1, \dots, u_{k-1})}_{L^\infty(\QT)}} T + \int_0^T \norm{(u_1, \dots, u_{k-1})}_{\hsp}^2(t) \, \dd t ,
	\end{align*}
	where we used Proposition \ref{prop:bahouri} from the appendix to bound the composition in Sobolev norm. Since by Sobolev embedding $\norm{u_i}_{L^\infty(\QT)} \lesssim \norm{u_i}_{X_T^s} \le  \norm{u_i}_{E_T^s}$,  all the quantities in the right-hand side of the above inequality can be bounded by a $C(T, \norm{U^0}_{\hs},\norm{F}_{Y_T^{s-1}})$ using the induction hypothesis. This concludes the proof. 
\end{proof}

We now turn to a stability estimate for \eqref{eq:SKTtriAffine}. 
\begin{Prop} \label{prop:triangStab}
	Let $U=(u_1, \dots, u_n)$ and $V=(v_1,\dots,v_n)$ be solutions of \eqref{eq:SKTtriAffine} with respective data $(U^0, F)$ and $(V^0, F')$, both in $H^s(\T^d) \times Y_T^{s-1}$. Then for all $1 \le k \le n$, $u_k - v_k$ satisfies
	$$
	\norm{u_k - v_k}_{E_T^s}^2
	\lesssim_{T,\norm{U}_{E_T^s}, \norm{V}_{E_T^s}} \norm{u^0_k - v^0_k}_{\hs}^2 + \norm{f_k' - f_k}_{Y_T^{s-1}}^2.
	$$ 
\end{Prop}
\begin{proof}
	Let $z_k = u_k -v_k$. We are in the setting of Corollary \ref{cor:dualStab}, which we apply to see
	$$\norm{z_k}_{E_T^s}^2 \lesssim  \exp(cT + c \norm{\mu_k(U)}_{Y_T^{s+1}}^2 )\left( \norm{z_k^0}_{\hs}^2 + \norm{f_k- f_k'}_{Y_T^{s-1}}^2 + \norm{\mu_k(U) -\mu_k(V) }_{E_T^{s}}^2\norm{v_k }_{E_T^{s}}^2 \right).$$
	Using Proposition \ref{prop:bahouri} from the appendix as in the previous proof, $\norm{\mu_k(U)}_{Y_T^{s+1}} \lesssim_{ \norm{U}_{X_T^{s}}}  \sqrt T + \norm{U}_{Y_T^{s+1}}$ and $\norm{\mu_k(U) -\mu_k(V) }_{E_T^{s}} \lesssim_{\norm{U}_{X_T^{s}},\norm{V}_{X_T^{s}}} \norm{(z_1, \dots, z_{k-1}) }_{E_T^{s}} $, so
	$$\norm{z_k}_{E_T^s}^2 \lesssim_{T,\norm{U}_{E_T^s}, \norm{V}_{E_T^s}}   \norm{z_k^0}_{\hs}^2 + \norm{f_k- f_k'}_{Y_T^{s-1}}^2 + \norm{(z_1, \dots, z_{k-1}) }_{E_T^{s}} .$$
	The last term can be bounded using an induction hypothesis.
\end{proof}

\subsection{Fixed point argument} \label{sec:fixedPoint} 

We are now interested in proving local existence and uniqueness for the full system \eqref{eq:SKTtr}. To do so, we use a fixed point procedure and let $\Phi:X_T^s \rightarrow E_T^s \supset X_T^s $ be the function that maps $U \in X_T^s$ to the solution $U^\star$ of
\begin{equation} \label{eq:defPhi}
	\left\{\begin{array}{ll}
		\partial_t u_k^*  - \Delta[\mu_k(u_1^*, \dots u_{k-1}^*) u_k^*]= r_k(u_1, \dots, u_n)  \quad \forall k \in\{1 ,\dots, n \} \\ 
		u_k^*(t=0, \cdot) = u_k^0
	\end{array} \right. 
\end{equation}
Let $U,V \in X_T^s$ be functions with the same initial value $U(t=0, \cdot) =V(t=0, \cdot) = U^0$. By Proposition \ref{prop:bahouri} of the appendix,
$$ \norm{R(U)}_{Y_T^{s-1}} \le \norm{R(U)}_{Y_T^{s}} \lesssim_{\norm{U}_{L^\infty(\QT)}} \sqrt T + \norm{U}_{Y_T^{s}},$$ 
so since $X_T^s \hookrightarrow Y_T^s \cap L^\infty(\QT)$, using Corollary \ref{cor:triBound} 
$$  \norm{\Phi(U)}_{E_T^{s}}^2 \le  C(T, \norm{U}_{X_T^s}),$$
where $C: \R_+^2 \rightarrow \R_+$ is increasing in both its arguments.
Applying the stability estimate from Proposition \ref{prop:triangStab} and then Proposition \ref{prop:bahouri} gives
$$
\norm{\Phi(V)  -\Phi(U)}_{E_T^{s}}^2 
\lesssim_{ \norm{\Phi(U)}_{E_T^{s}}, \norm{\Phi(V)}_{E_T^{s}}}  \norm{R(U) - R(V)}_{Y_T^{s-1}}^2 \lesssim_{T,  \norm{U}_{X_T^s},  \norm{V}_{X_T^s}}  \norm{U-V}_{Y_T^{s}}^2  .
$$
Since $\norm{\cdot}_{Y_T^s}^2 \le T\norm{\cdot}_{X_T^s}^2 $, there exists a function $C_0: \mathbb \R_+^3 \rightarrow \R_+$ increasing in all of its arguments such that 
\begin{equation*} 
	\norm{\Phi(V)  -\Phi(U)}_{E_T^{s}}^2  \le T C_0(T,  \norm{U}_{X_T^s},  \norm{V}_{X_T^s}) \norm{U-V}_{X_T^s}^2
\end{equation*} 	
In order to apply Banach's fixed point theorem, we search for $R$ and $T$ such that the closed ball $\overline {B_{X_T^s}(0,R)}$ is stable by $\Phi$ and such that $\Phi$ is contractant on it. For the  contractivity it is sufficient to have $ \sqrt T C_0(T,R,R) < 1$, which is true for all $T$ small enough. For the stability constraint, we compare $U \in \overline {B_{X_T^s}(0,R)}$ with the constant in time function $U_0$,
$$
\norm{\Phi(U)  -\Phi(U^0)}_{E_T^{s}}
\le   \sqrt T C_0(T,R,\norm{\Phi(U_0)}_{X_T^s}) \norm{U-U^0}_{X_T^s} .
$$
Therefore, adding the constraint  $R > \norm{\Phi(U^0)}_{E_T^s}$ (which implies $ R >  \norm{U^0}_{\hs}$),
$$ \norm{\Phi(U)}_{E_T^{s}}
\le R + \sqrt T C_0 (T,R,R) \times 2R .$$
If we begin by fixing $R >  \norm{U^0}_{\hs}$ and choose $T$ small enough to have both $ R > \norm{\Phi(U^0)}_{E_T^s}$ and 
$$ R + 2 R\sqrt T  C_0(T,R,R) < R ,$$
which implies automatically $ \sqrt T C_0(T,R,R) < 1$, we get a ball $\overline {B_{X_T^s}(0,R)}$ on which we can apply Banach's fixed point theorem and find a unique fixed point of $\Phi$. This proves local existence.

For uniqueness, let $U$ and $V$ be two solutions with the same initial condition $U^0$ and let $T^*$ be the shortest of the two lifetimes of $U$ and $V$. Let $T_0 = \sup \{T < T^* \mid U_{|[0,T] } = V_{|[0,T] }\}$. We prove $T_0=T^*$ by contraction. If $T_0 < T$, we prove that there exists $\varepsilon > 0$ such that $U$ and $V$ coincide until time $T_0 + \varepsilon$. By continuity of $U$ and $V$, if $T_0 < T^*$ then $U(T_0) = V(T_0)$. Therefore, the proof reduces to the case $T_0 = 0$. 
In this case, we set $T < T^*$,  $R > \max(\norm{U}_{X_T^{s}}, \norm{V}_{X_T^{s}})$ and $\varepsilon$ small enough so that $\Phi$ is contractant on $\overline {B_{X_\varepsilon^s}(0,R)}$. Since $U$ and $V$ are both fixed points of $\Phi$ in $\overline {B_{X_\varepsilon^s}(0,R)}$, they coincide until time $\varepsilon$. We proved that the coincidence duration of $U$ and $V$ can be improved by $\varepsilon$, hence the contradiction. Therefore, $U$ and $V$ coincide as long as both solutions live, proving uniqueness of solutions of \eqref{eq:SKTtr}.

Looking back to the local existence proof, we found for all initial condition $U^0$ a time $T_{\norm{U^0}_{X_T^s}}$ of existence which can be expressed as a decreasing function of $\norm{U^0}_{X_T^s}$. If the solution stays bounded in $X_T^s$ by some $R>0$, we can extend the existence interval and get existence until $t+T_R$ for all $t$ such that $\norm{U(t)}_{X_T^s} < R$. This yields the following explosion criterion.
\begin{Prop} \label{prop:localExist}
	For all $ U^0 \in H^s (\T^d)$, there exists $T>0$ such that there exists a unique solution $U \in E_T^s $ of \eqref{eq:SKTtr}.   
	If $T^\star$ is the maximal lifetime of the solution, then either $T^* = + \infty $ or
	$$ \norm{U(t)}_{\hs}  \underset{t \to T^*}{\longrightarrow} \infty .$$
\end{Prop}

\subsection{Explosion criteria, proof of Proposition \ref{prop:exploLp}} \label{sec:explo}

In this section, we prove the following refinement of Proposition \ref{prop:localExist}:
\begin{Prop} \label{redLinfty} 
	For all $ U^0 \in H^s (\T^d)$, the maximal lifetime $T^*$ of the solution $U \in E_T^s $ from Proposition \ref{prop:localExist} satisfies either $T^* = + \infty $ or
	$$ \norm{U(t)}_{L^\infty(\T^d)} \underset{t \to T^*}{\longrightarrow} \infty .$$
\end{Prop}

\begin{proof}
	
	We first observe that for all $ k \in\{1 ,\dots, n \}  $, by Proposition \ref{prop:bahouri} of the appendix we have the following two bounds
	$$\norm{\mu_k(u_1, \dots, u_{k-1})}_{Y_T^{s+1}} \lesssim_{\norm{  U}_{L^\infty(\QT)}} \sum_{i=1}^{k-1}\norm{u_i}_{Y_T^{s+1}} , $$
	$$\norm{r_k(U)}_{Y_T^{s-1}} \le \norm{r_k(U)}_{Y_T^{s}} \lesssim_{\norm{  U}_{L^\infty(\QT)}} \norm{U}_{Y_T^{s}}, $$
	where the constant behind $\lesssim_{\norm{  U}_{L^\infty(\QT)}}$ depends only on the parameters of the problem and $\norm{  U}_{L^\infty(\QT)}$.
	Thus, applying Corollary \ref{dualiteHsalt} with $\varepsilon=0$ one gets
	$$ \norm{u_k}_{E_T^s}^2 \lesssim_{T, \norm{  U}_{L^\infty(\QT)}} 1+  \norm{u_k^0}_{\hs}^2 + \norm{U}_{Y_T^{s}}^2   + \sum_{i=1}^{k-1}\norm{u_i}_{Y_T^{s+1}}^2.$$
	Exploiting the triangular structure of this system, we deduce by induction on $k$ that
	$$ \norm{U}_{E_T^s}^2 \lesssim_{T,\norm{  U}_{L^\infty(\QT)}}   \norm{U^0}_{\hs}^2 + \norm{U}_{Y_T^{s}}^2 .$$
	Grönwall's lemma then allows to absorb the integral term $\norm{U}_{Y_T^{s}}^2 $ from the right-hand side in the uniform term $\norm{U}_{X_T^s}^2$ contained in the left-hand side and write
	$$ \norm{U}_{E_T^s}^2 \lesssim_{T,\norm{  U}_{L^\infty(\QT)}}   \norm{U^0}_{\hs}^2 + 1 .$$
	Now, assume the lifetime $T^*$ of $U$ is finite. Then, by Proposition \ref{prop:localExist},
	$ \norm{U}_{E_T^s} \underset{T \to T^*}{\longrightarrow} \infty $ and the previous estimate implies explosion of $\norm{  U}_{L^\infty(\QT)}$ when $T \rightarrow T^*$.
	
\end{proof}

We now turn to the proof of the refined explosion criterion in Proposition \ref{prop:exploLp} when the parameters $(\mu_k)_{1  \le k \le n}$ and $(r_k)_{1  \le k \le n}$ have enough derivatives that are bounded by a polynomial. The proof relies on the following refined tame estimate for composition in Sobolev spaces, which is an instanciation of Proposition \ref{prop:wekaerComp} in the appendix.

\begin{Prop} \label{prop:wekaerCompInst}
	Let $s \ge 0$,  $\varepsilon>0$ and $f\in  \mathcal C^{\infty}(\R^n; \R) $ such that there exists constants $C,\alpha>0$ satisfying $|\nabla^{\lceil s \rceil}f(X)| \le C(1 + |X|^\alpha)$. Then, there exists $p<\infty$ and a polynomial function $P$ such that for all $U \in H^{s + \varepsilon}(\T^d) $,
	$$\norm{f(U)}_{H ^{s}}  \le P(\norm U _{L ^p} ) (1 + \norm U _ {H ^{s+ \varepsilon}} ). $$
\end{Prop}

Having in mind the example of logistic reaction terms, we prove that when , the explosion criterion can be refined:

\begin{proof}[Proof of Proposition \ref{prop:exploLp}]  Let $s_0 \in (\frac d 2, \min(s,\frac d 2 +1 ))$ and $\varepsilon >0 $ such that $s_0- \varepsilon n > \frac d 2$ and set $s_k = s- \varepsilon k$. The assumption on $\mu_k$ allows to apply Proposition \ref{prop:wekaerCompInst} to get the existence of $p < \infty$ and a polynomial $P: \R^2 \rightarrow \R_+$  such that for all $ k \in\{1 ,\dots, n \}  $, 
	\begin{equation*}
		\norm{\mu_k(u_1, \dots, u_{k-1})}_{Y_T^{s_k+1 + \varepsilon /2}}^2  \le  \int_0^T P(\norm{  U(t)}_{L^p} )^2 \left( 1+  \sum_{i=1}^{k-1}\norm{u_i(t)}_{H^{s_k+1 + 3\varepsilon /4}} ^2 \right) \dt .
	\end{equation*} 
	We also bound the reaction terms using Proposition \ref{prop:wekaerCompInst} and the fact that $s_k -1 < s_n$ for all $ 1 \le k \le n$:
	$$\norm{r_k(U)}_{Y_T^{s_k-1}}^2 \le \int_0^T P(\norm{  U(t)}_{L^p} )^2 ( 1 +\norm{U(t)}_{H^{s_n}} ^2) \dt.$$
	Plugging these bounds in the estimate from Corollary \ref{dualiteHsalt} and using that $s_k > s_n$, we get the existence of $p < \infty$  and a polynomial $P$ such that 
	\begin{align}
		\norm{u_k}_{E_T^{s_k}}^2 &\lesssim \norm{u_k^0}_{H^{s_k}}^2+ \int_0^T \mean{P(u_k(t))} \norm{u_k(t)}_{H^{s_k}}^2\, \dd t  \nonumber
		\\
		&\hspace{3.5cm} + \int_0^T P(\norm{  U(t)}_{L^p} )^2 \left( 1+ \norm{U(t)}_{H^{s_n}} ^2 + \sum_{i=1}^{k-1}\norm{u_i(t)}_{H^{s_k+1 + 3\varepsilon /4}} ^2 \right) \, \dd t.  \nonumber \\
		& \lesssim \norm{U^0}_{H^s}^2
		+ \int_0^T P(\norm{  U(t)}_{L^p} )^2 \left( 1+ \sum_{i=1}^{n}\norm{u_i(t)}_{H^{s_i}} ^2 + \sum_{i=1}^{k-1}\norm{u_i(t)}_{H^{s_k+1 + 3\varepsilon /4}} ^2 \right) \, \dd t.   \label{eq:LpExp1}
	\end{align}
	Let us bound the last integral. Up to replacing $p$ by $\max( p, 2 \deg P)$ we have 
	$$ \int_0^T P(\norm{  U(t)}_{L^p} )^2 \dt \lesssim T + \norm {U}_{L^p(\QT)}^p \lesssim_{ T, \norm {U}_{L^p(\QT)}} 1 . $$
	On the other hand, letting $r >1$ to be fixed later and using Hölder's inequality,
	\begin{equation} \label{eq:LpExp2}
		\int_0^T P(\norm{  U(t)}_{L^p} )^2 \norm{u_i(t)}_{H^{s_k+1 + 3\varepsilon /4}} ^2 \dt \le \left(\int_0^T P(\norm{  U(t)}_{L^p} )^{2r'} \dt \right)^{\frac 1 {r'}} \norm{u_i}_{L^{2r}_T H^{s_k+1 + 3\varepsilon /4}}^2.
	\end{equation} 
	Up to choosing $p$ large enough, the integral in the right-hand side can be bounded by an increasing function of $T$ and $\norm {U}_{L^p(\QT)}$.
	The assumptions on $(s_i)_{1 \le i  \le n}$ imply that $s_{k-1} + 1 > s_{k} + 1 + \frac 3 4 \varepsilon $ and $s_{k} + 1 + \frac 3 4 \varepsilon > \frac d 2 +1 > s_{k-1} $ so there exists $\theta \in (0,1)$ such that
	$$  s_{k} + 1 + \frac 3 4 \varepsilon = \theta (s_{k-1} + 1) + (1 - \theta) s_{k-1}.$$
	Setting $r = \frac 1 \theta $, we have $ \frac 1 {2r} =  \frac \theta 2 + \frac{1 - \theta} {\infty}$ so $L^{2r}_T H^{s_k+1 + 3\varepsilon /4}$ can be interpolated between $X_T^{ s_{k-1}}$ and $Y_T^{ s_{k-1}}$ so for all $1 \le i\le k-1$,
	$$ \norm{u_i}_{L^{2r}_T H^{s_k+1 + 3\varepsilon /4}} \lesssim   \norm{u_i}_{Y_T^{s_{k-1}+1}}^{\theta}  \norm{u_i}_{X_T^{s_{k-1}}}^{1- \theta} \lesssim \norm{u_i}_{E_T^{s_{k-1}}} \lesssim \norm{u_i}_{E_T^{s_{i}}}. $$
	Plugging this in \eqref{eq:LpExp2} and then in \eqref{eq:LpExp1}, we get for $p < \infty$ large enough
	\begin{equation*} 
		\norm{u_k}_{E_T^{s_k}}^2 \lesssim_{T, \norm {U}_{L^p(\QT)}} \norm{U^0}_{H^s}^2 + \sum_{i=1}^{n} \int_0^T (1+ \norm{  U(t)}_{L^p}^p ) \norm{u_i(t)}_{H^{s_i}} ^2 \, \dd t + \sum_{i=1}^{k-1} \norm{u_i}_{E_T^{s_{i}}}^2 .
	\end{equation*} 
	Using the triangular structure of this system and then using Grönwall's lemma like in the proof of the previous proposition we deduce
	$$  \sum_{k=1}^n \norm{u_k}_{E_T^{s_k}}^2 \lesssim_{T, \norm {U}_{L^p(\QT)}}   1+  \norm{U^0}_{\hs}^2  .$$
	In particular, since $s_k \ge s_n > \frac d 2$, we have a bound on $\norm{  U}_{L^\infty(\QT)}$ depending only on $T,  \norm {U}_{L^p(\QT)}$ and $ \norm{U^0}_{\hs}$.
	Now, assume the lifetime $T^*$ of $U$ is finite. Then, by the $L^\infty(\QT)$ explosion criterion (Proposition \ref{redLinfty}) 
	$ \norm{U}_{L^\infty(\QT)} \underset{T \to T^*}{\longrightarrow} \infty $, so we must also have explosion of $ \norm {U}_{L^p(\QT)}$ when $T \rightarrow T^*$.
	
\end{proof}
\begin{Rmk}
	\begin{enumerate}[label=(\roman*)]
		\item Using a Sobolev embedding, we deduce that under the hypotheses of Proposition \ref{prop:exploLp} there exists $s < \frac d 2$ such that if a solution blows up, then the blow-up happens in $\norm{\cdot}_{\hs}$ norm.
		\item It is clear from the above proof that if we already know that $\norm{u_1}_{L^\infty(\QT)}, \dots, \norm{u_k}_{L^\infty(\QT)} $ do not blow up, then the polynomial bound assumption on $\mu_j$ can be removed for $j \le k+1$.
	\end{enumerate}
\end{Rmk}

\subsection{Smoothness and sign preservation } \label{sec:end}

To complete the proof of Theorem \ref{th:localExist}, there remains to prove the smoothness of the solutions and the stability estimate. Under Hypotheses \ref{ass:1}—\ref{ass:4}, let $U^0 \in H^s(\T ^d) $ and $U \in E_T^s $ be the associated solution, with maximal lifetime $T^*$. By Sobolev embedding, to prove the smoothness it is sufficient to prove that for all $k\in \mathbb N$ and $r> d/ 2$, $U \in \mathcal C^k (  (0, T^*); H ^r(\T^d))$. \\

\emph{Space smoothness.} We first prove that $U \in \mathcal C (  [0, T^*); H ^s(\T^d))$. Let $\tilde T$ be the maximal time for which $U \in \mathcal C (  [0, \tilde T); H ^s(\T^d))$. By Proposition \ref{prop:localExist}, $\tilde T > 0$. Suppose that $\tilde T < T^*$ and let $R = \sup_{t \le \tilde T} \norm{U(t)}_{\hs}$. By the explosion criterion, $R$ is finite and the construction used to prove Proposition \ref{prop:localExist} builds a time $T_R$ such that for all $t < \tilde T$, $U$ exists up to time $t+ T_R$ and $U \in  \mathcal C (  [t, t+ T_R); H ^s(\T^d))$. Choosing $t \in (\tilde T - T_R, \tilde T)$, one obtains a contradiction with the maximality of $\tilde T$. Therefore, $\Tilde T = T^*$ and $U \in \mathcal C (  [0, T^*); H ^s(\T^d))$.

We now want to deduce by induction that for all $r> d/ 2$, $U \in \mathcal C (  (0, T^*); H ^r(\T^d))$. Let $r> d/ 2$ such that $U \in \mathcal C (  (0, T^*); H ^r(\T^d))$. Then in particular, for all $t_0>0$, $U(t_0) \in H^r(\T^d)$ and Proposition \ref{prop:localExist} applied with initial time $t_0$ and in the space $E_T^r$ implies that for some $\delta > 0$, $U \in L ^2 ([t_0, t_0 + \delta]; H^{r+1}(\T^d))$. Thus, there exists $\varepsilon > 0$ such that $U(\varepsilon) \in H^{r+1}(\T^d)$. The time $\varepsilon$ can be chosen arbitrarily small since it is also the case of $t_0$. Since $U$ is continuous is time, $U(\varepsilon)$ is a well-defined function and it can be used as an initial condition for the Cauchy problem \eqref{eq:SKTtr}. Therefore, the discussion at the beginning of the proof of space smoothness implies that $U \in \mathcal C (  [\varepsilon, T_{r+1}); H ^{r+1}(\T^d))$, where $T_{r+1}$ is the maximal lifetime of the solution in $H^{r+1}(\T^d)$. But the explosion criterion from Proposition \ref{redLinfty} states that $T_{r+1}$ is actually the lifetime of the solution in $L^\infty(\T^d)$, which is $T^*$. Since moreover $\varepsilon$ can be chosen arbitrarily small, we proved $U \in \mathcal C (  (0, T^*); H ^{r+1}(\T^d))$.  \\

\emph{Time smoothness.} For time smoothness we also proceed by induction and suppose that for some $k\in \mathbb N$, all $r> d/ 2$ satisfy $U \in \mathcal C^k (  (0, T^*); H ^r(\T^d))$. Note that the previous part of the proof serves as the initial step of the induction. Let $r> d/ 2 $. By Proposition \ref{prop:bahouri} from the appendix, for all $i \in \{1, \dots  n\}$,
$$ \mu_i(U) u_i \in \mathcal C^k (  (0, T^*); H ^{r+2}(\T^d)) \qquad \mathrm{ and } \qquad R_i(U) \in \mathcal C^k (  (0, T^*); H ^r(\T^d)).$$
Reading equation \eqref{eq:SKTtr}, this implies $\partial_t u_i \in \mathcal C^k (  (0, T^*); H ^r(\T^d))$, hence $u_i \in \mathcal C^{k+1} (  (0, T^*); H ^r(\T^d))$. This concludes the proof of smoothness. \\

\emph{Stability estimate.} Let $U$ and $V$ be two solutions of \eqref{eq:SKTtr}. Applying the stability estimate from Proposition \ref{prop:triangStab} and then Proposition \ref{prop:bahouri} gives
\begin{align} 
	\norm{U-V}_{E_T^{s}}^2 
	&\lesssim_{ T, \norm{U}_{E_T^{s}}, \norm{V}_{E_T^{s}}} \norm{U^0 - V^0}^2 _{\hs} +  \norm{R(U) - R(V)}_{Y_T^{s-1}}^2 \nonumber \\
	&\lesssim_{T,  \norm{U}_{E_T^s},  \norm{V}_{E_T^s}} \norm{U^0 - V^0}^2_{\hs} + \norm{U-V}_{Y_T^{s}}^2  . \label{eq:refstab}
\end{align}
This inequality contains 
$$  \norm{U(T) - V(T)}^2_{\hs} \lesssim_{T,  \norm{U}_{E_T^s},  \norm{V}_{E_T^s}} \norm{U^0 - V^0}^2_{\hs} + \int_0^T \norm{U(t) - V(t)}^2_{\hs} \, \dd t, $$
so by Grönwall's lemma $\norm{U-V}_{E_T^{s}}^2 \lesssim_{T,  \norm{U}_{E_T^s},  \norm{V}_{E_T^s}}  \norm{U^0 - V^0}^2_{\hs}$, which is the wanted stability estimate. This concludes the proof of Theorem \ref{th:localExist}. \\

\emph{Sign preservation.} The argument to prove sign preservation under \ref{ass:4} is the same as the one used to prove Proposition 1.4 in \cite{gallagher_cauchy_2025}. We send the reader to this article for more details.

\section{Proof of Theorem \ref{th:global}} \label{sec:5}

In this section, we assume \ref{ass:1}–\ref{ass:4} and \ref{ass:6}, $d \le 2$, $n =2$ and $U^0 \in H^s(\T^d)$. In this setting, \eqref{eq:SKTtr} rewrites
\begin{equation} \label{eq:reacSimp}
	\left\{\begin{array}{ll}
		\partial_t u_2  - \Delta[\mu_2(u_1) u_2]= u_2\bar r_2(u_1,u_2) \\
		\partial_t u_1  - \mu_1\Delta u_1= u_1\bar r_1(u_1,u_2)\\
		(u_1,u_2)(t=0, \cdot) = (u_1^0, u_2^0)
	\end{array} \right. 
\end{equation} 
\\
Let $T>0$ be a time of existence of the solution $U$. We will allow the constant behind $\lesssim $ to also depend on $U^0$ and $T$. For example, we can now write $\norm{U^0}_{\hs} + T \lesssim 1$. We already have an $L^\infty(\QT)$ bound on $u_1$ since the maximum principle (see \cite[Theorem 6.4.2]{evans_partial_1988} for a detailed discussion of maximum principles) and the positivity of the solutions proved in Theorem \ref{th:localExist} together give
$$ \norm{u_1}_{L^\infty(\QT)} \le \exp( (\sup \bar r_1) t ) \norm{u_1^0}_{L^\infty} \lesssim 1 .  $$

According to the remarks following the proof of Proposition \ref{prop:exploLp}, no polynomial bound assumption on $\mu_2$ is needed to apply it. Therefore, \ref{ass:6} is sufficient to apply Proposition \ref{prop:exploLp} and we now have to prove $\norm{U}_{L^p(\QT)} \lesssim 1$ for all $p < \infty$. Since $u_1$ is already bounded in $L^\infty(\QT)$, there remains to bound $\norm{u_2}_{L^p(\QT)}$.
Using the Sobolev embedding $ H^1(\T^d) \hookrightarrow L^p(\T^d)$ for all $d\le 2$ and $p< \infty$,
$$ \norm{u_2}_{L^p(\QT)} \lesssim  \norm{u_2}_{L^\infty_TL^p} \lesssim  \norm{u_2}_{X_T^1} \lesssim \norm{u_2}_{E_T^1}. $$
Therefore, it is sufficient to prove $ \norm{u_2}_{E_T^1} \lesssim 1$ to conclude. We do so by starting with low-regularity bounds and improve them iteratively until obtaining a bound on $\norm{u_2}_{E_T^1}$. Note that since the solutions built by Theorem \ref{th:localExist} become instantaneously smooth, up to slightly moving the origin of time we can assume $U^0$ to be smooth. This will ensure that all the quantities in which $U^0$ appears are finite.

\emph{Step 1. Bound in $E_T^{-1}$.} The duality lemma \ref{lem:duality-1} gives
\begin{multline*}
	\norm{u_2(T)}_{H^{-1}}^2 + \int_\QT \mu_2(u_1) u_2^2 \\
	\lesssim \norm{u_2^0}_{H^{-1}}^2 + \norm{\mu_2(u_1)}_{L^\infty(\QT)}\int_0^T \norm{u_2(t)}_{H^{-1}}^2 \, \dd t + \int_0^T \norm{r_2(u_1(t),u_2(t))}_{H^{-2}}^2 \, \dd t.
\end{multline*} 
Since $u_2$ is non-negative and $r_2(u_1,u_2) \le  u_2  \sup \bar r_2$, Lemma \ref{lem:duality-1} allows to replace $r_2(u_1,u_2) $ by $ u_2  \sup \bar r_2$ in the above estimate and get
$$\norm{u_2(T)}_{H^{-1}}^2 + \int_\QT \mu_2(u_1) u_2^2 \lesssim 1 + \int_0^T (\norm{u_2(t)}_{H^{-1}}^2 + \norm{u_2(t)}_{H^{-2}}^2) \, \dd t. $$
Since $\norm{\cdot}_{H^{-2}} \le\norm{\cdot}_{H^{-1}}  $ and $ \inf \mu_2 >0$, using Grönwall's lemma yields the bound $\norm{u_2}_{L^2(\QT)} \lesssim 1$. Using an energy estimate on $u_1$ and the bound $|\bar r_1(X)| \lesssim 1 +|X|$, this last bound implies 
$$ \norm{u_1}_{E_T^1} \lesssim \norm{u_1 \bar r_1(u_1,u_2)}_{L^2(\QT)} \le \norm{u_1}_{L^\infty(\QT)} \norm{\bar r_1(u_1,u_2)}_{L^2(\QT)} \lesssim 1 + \norm{u_1}_{L^2(\QT)} + \norm{u_2}_{L^2(\QT)} \lesssim 1.$$
The interpolation inequality $ \norm{u_1(t)}_{\frac 3 2} \lesssim \norm{u_1(t)}_{1}^{\frac 1 2}\norm{u_1(t)}_{2}^{\frac 1 2}  $ then implies a bound of $u_1$ in $L^4([0,T]; H^{\frac 3 2}(\T^d))$. Hence, by the Sobolev embedding $H^{\frac 1 2}(\T^d) \hookrightarrow L^4(\T^d)$ in dimension $d\le 2$, $ \norm{\nabla u_1}_{L^4(\QT)} \lesssim 1$.

\emph{Step 2. Bound in $E_T^0$.} We multiply the equation on $u_2$ by $u_2$ itself and integrate by parts to get
$$ \frac 1 2 \norm{u_2(T)}_{L^2}^2 + \int_\QT \mu_2( u_1) |\nabla u_2|^2 =  \frac 1 2 \norm{u_2^0}_{L^2}^2 - \int_\QT u_2 \nabla u_2 \cdot \nabla (\mu_2(u_1)) + \int_\QT u_2^2 \bar r_2(u_1, u_2). $$ 
Using Hölder's inequality and the fact that $\inf \mu_2 > 0$,
$$ \norm{u_2(T)}_{L^2}^2 + \int_\QT  |\nabla u_2|^2 \lesssim 1 +  \int_0^T \norm{u_2(t)}_{L^4} \norm{\nabla u_2(t)}_{L^2} \norm{ \nabla (\mu_2(u_1(t)))}_{L^4} \, \dd t   + \int_\QT u_2^2.$$
Since $\nabla (\mu_2(u_1)) = \mu_2'(u_1)\nabla u_1$, $\mu_2' $ is continuous and $\norm{u_1}_{L^\infty(\QT)} \lesssim 1$, we have $\norm{ \nabla (\mu_2(u_1(t)))}_{L^4} \lesssim \norm{ \nabla u_1(t)}_{L^4}$. Using additionally the bound $\norm{u_2(t)}_{L^4} \lesssim \norm{u_2(t)}_{\frac 1 2} \lesssim \norm{u_2(t)}_{L ^2}^{\frac 1 2}\norm{u_2(t)}_{1}^{\frac 1 2}$,
$$ \int_0^T \norm{u_2(t)}_{L^4} \norm{\nabla u_2(t)}_{L^2} \norm{ \nabla (\mu_2(u_1(t)))}_{L^4} \, \dd t \lesssim \int_0^T \norm{u_2(t)}_{L^2}^{\frac 1 2} \norm{u_2(t)}_{1}^{\frac 3 2} \norm{ \nabla u_1(t)}_{L^4} \, \dd t$$
Therefore, applying Young's inequality and injecting the previous estimate in the above one, for all $\varepsilon > 0$,
$$ \norm{u_2(T)}_{L^2}^2 + \int_\QT  |\nabla u_2|^2 \lesssim 1 +   C_\varepsilon \int_0^T \norm{u_2(t)}_{L^2}^{ 2} \norm{ \nabla u_1(t)}_{L^4}^4 \, \dd t +  \varepsilon \int_0^T \norm{u_2(t)}_{1}^{2}\, \dd t + \int_0^T \norm{u_2(t)}_{L^2}^2 \, \dd t,$$
where $C_\varepsilon$ is a constant depending only on $\varepsilon$. Up to adding $\norm{u_2}_{Y_T^0}^2$ on both sides, we can assume that we actually have $\norm{u_2}_{Y_T^1}^2$ on the left-hand side. Thus, fixing $\varepsilon$ small enough to absorb the $\norm{u_2}_{Y_T^1}^2$ from the right-hand side,
$$ \norm{u_2(T)}_{L^2}^2 + \norm{u_2}_{Y_T^1}^2 \lesssim 1 +  \int_0^T \norm{u_2(t)}_{L^2}^{ 2} \norm{ \nabla u_1(t)}_{L^4}^4 \, \dd t  + \int_0^T \norm{u_2(t)}_{L^2}^2 \, \dd t,$$
and by Grönwall's lemma
$$  \norm{u_2}_{E_T^0}^2 \lesssim \exp \left(T + \int_0^T \norm{ \nabla u_1(t)}_{L^4}^4 \, \dd t \right).$$
This is sufficient to deduce $\norm{u_2}_{E_T^0}\lesssim 1$ since we already proved $\norm{ \nabla u_1}_{L^4(\QT)} \lesssim 1$ at the end of Step 1.
Using once again the Sobolev embedding $H^{\frac 1 2} \rightarrow L ^4$ and interpolation, we have  $\norm{u_2}_{L^4(\QT)}\lesssim \norm{u_2}_{E_T^0}\lesssim 1$. By parabolic regularity (see for example \cite{salamon_parabolic_2017}), we also have
$$ \norm{\Delta u_1}_{L^4(\QT)} \lesssim \norm{ \Delta u_1^0 }_{L^4} +  \norm{u_1 \bar r_1(u_1, u_2)}_{L^4(\QT)} \lesssim  \norm{ u_1^0 }_{H^{\frac 5 2}} + \norm{1+ u_2}_{L^4(\QT)} \lesssim 1. $$

\emph{Step 3. Bound in $E_T^1$ and conclusion.} We can once again improve the bound on $u_2$, this time by applying Proposition \ref{prop:dualiteHs}:
·\begin{equation} \label{eq:1749}
	\norm{u_2}_{E_T^1}^2 \lesssim   \norm{u_2^0}_{1}^2+ \int_0^T \norm{u_2(t)}_{1}^2\, \dd t + \norm{u_2  \bar r_2(u_1,u_2)}_{Y_T^{0}}^2 + \int_0^T \norm{[\mu_2(u_1(t)), J^{2}] u_2(t)}_{L^2}^2\, \dd t.
\end{equation} 
Let's examine the terms on the right-hand side. The first one is bounded and the second one can be treated using Grönwall's lemma. For the third one, notice that $Y_T^0 = L^2(\QT)$ and
$$ \norm{u_2  \bar r_2(u_1,u_2)}_{L^2(\QT)}^2 \lesssim \int_\QT u_2^2(1 + u_1^2 + u_2^2) \lesssim 1.  $$ 
For the last term, we compute the commutator by remembering that by definition of $J$, $J^2= I - \Delta$:
$$[\mu_2(u_1), J^{2}] u_2 = -[\mu_2(u_1), \Delta] u_2 = u_2 \Delta[\mu_2(u_1)] + 2 \nabla u_2 \cdot \nabla [\mu_2(u_1)].  $$
Using Hölder's inequality and the chain rule, we get
$$ \norm{u_2 \Delta \mu_2(u_1)}_{L^2(\QT)} \le
\norm{u_2}_{L^4(\QT)}\norm{\Delta\mu_2(u_1)}_{L^4(\QT)} \lesssim \norm{u_2}_{L^4(\QT)} \norm{\Delta u_1}_{L^4(\QT)} \lesssim 1.$$
There remains to bound the last term, using once again the chain rule
$$\norm{\nabla u_2 \cdot \nabla [\mu_2(u_1)]}_{L^2(\QT)} \le \int_0^T \norm{\nabla u_2(t)}_{L^2}^2\norm{\nabla u_1(t)}_{L^\infty}^2 \, \dd t. $$
We can absorb this term using Grönwall's lemma since $ \norm{u_2}_{X_T^1}^2 $ appears in the left-hand side of \eqref{eq:1749}. This would conclude the proof of $\norm{u_2}_{E_T^1} \lesssim 1 $ if 
$\int_0^T\norm{\nabla u_1(t)}_{L^\infty}^2 \, \dd t \lesssim 1.$
To prove this final estimate, remark that by Sobolev embedding, we have for all $t$ that $\norm{\nabla u_1(t)}_{L^\infty} \lesssim \norm{u_1(t)}_{L^\infty} + \norm{\Delta u_1(t)}_{L^4}$. Therefore,
$$ \int_0^T\norm{\nabla u_1(t)}_{L^\infty}^2 \, \dd t \lesssim T + \int_0^T\norm{\Delta u_1(t)}_{L^4} \, \dd t \lesssim 1 + \norm{\Delta u_1}_{L^4(\QT)}.$$
Since we already proved $\norm{\Delta u_1}_{L^4(\QT)} \lesssim 1$ at the end of Step 2, this concludes the proof.

\section{Appendix}

\subsection{A transference principle} \label{sec:transfer}

In this appendix, we give an elementary proof of a transference principle for commutator estimates and Gagliardo-Nirenberg type inequalities from $\R^d$ to the periodic domain $\T^d$. Our main motivation is to transfer the following Kato-Ponce type estimate proved by Li in  \cite[Theorem 1.9]{li_kato-ponce_2016}.
\begin{Prop} \label{prop:liComutR}
	Let $s\ge 0$ and $r, p,\bar p, q, \bar q \in [2, \infty]$ such that $\frac 1 p + \frac 1 {\bar p} =\frac 1 q + \frac 1 {\bar q} = \frac 1 r$. Then for any $u,v \in \mathcal S(\T^d)$, the following estimate holds:
	$$\norm{[J^{s}, u]v}_{L^{r}(\R^d)} \lesssim \norm{J^{s-1} \nabla u}_{L^{\bar p}(\R^d)} \norm{v}_{L^p(\R^d)} + \mathbf 1 _{s > 1} \norm{ \nabla u}_{L^{\bar q}(\R^d)} \norm{J^{s-2}\nabla v}_{L^{q}(\R^d)},$$
	where $ \mathbf 1 _{s > 1}$ indicates that the second term is present only if $s > 1$.
\end{Prop}

The Fourier transform (resp.\ Fourier coefficients) of a distribution $ T \in \mathcal S'(\R^d)$  (resp.\ $ \mathcal D'(\T^d)$) will be denoted $\hat T$. We call Fourier multiplier any operator $L :  \mathcal S(\R^d) \rightarrow  \mathcal S'(\R^d)$ such that there exists a function $\hat L: \R^d \rightarrow \mathbb C$ satisfying $\widehat{Lf} = \hat L \hat f$. The function $\hat L$ is called the \emph{symbol} of $L$. 

Even though there already exists a transference theory for multilinear Fourier multipliers \cite{fan_transference_2001}, we require a slight generalization that allows an upper bound with the same structure as the one in Proposition \ref{prop:liComutR}. Our result is the following.


\begin{Lemma} \label{lem:transference}
	Let $L_0, \dots , L_{4}$ be Fourier multipliers whose symbols $\hat L_0, \dots, \hat L_{4} $ are continuous at each point of $\Z^d$ and bounded by a polynomial. \begin{enumerate}[label=(\roman*)]
		\item Let $p, p_1, p_2 \in [2, \infty]$ and $\theta \in [0,1]$ such that $\frac 1 p = \frac \theta {p_1} +  \frac {1-\theta} {p_2} $ and assume that the following estimate holds for all $u \in \mathcal S (\R^d)$:
		$$ \norm{L_0 u }_{L^p(\R^d)}  \lesssim  \norm{L_{1} u }_{L^{p_1}(\R^d)}^{\theta}   \norm{L_{2} u }_{L^{ p_2}(\R^d)}^{1 - \theta} .$$ 
		Then, replacing the domain of the norms by $\T^d$, the estimate holds for all $u \in \mathcal D (\T^d)$.
		\item Let $r, p,\bar p, q, \bar q \in [2, \infty]$ such that $\frac 1 p + \frac 1 {\bar p} = \frac 1 q + \frac 1 {\bar q}  = \frac 1 r$ Assume that the following estimate holds for all $u,v \in \mathcal S (\R^d)$:
		$$\norm{[L_0, u]v}_{L^{r}} \lesssim \norm{L_{1} u}_{L^{\bar p}} \norm{L_{2} v}_{L^p} + \norm{L_{3} u}_{L^{\bar q}} \norm{L_{4} v}_{L^{q}}.$$
		Then, replacing the domain of the norms by $\T^d$, the estimate holds for all $u,v \in \mathcal D (\T^d)$.
	\end{enumerate}

\end{Lemma}

Distributions on $\T^d$ will be identified with the $\Z^d$-periodic distributions over $\R^d$. The identification mapping $\mathcal D'(\T^d) \rightarrow \mathcal S'(\R^d)$ can be seen as the dual of the wrapping map, which associates to any $f \in \mathcal S (\R^d)$ the $\mathcal D(\T ^d )$ function
$$ x \mapsto \sum_{k \in \Z^d} f(x + k).$$

Let $T \in \mathcal D'(\T^d)$ and $T' \in  \mathcal S'(\T^d)$ be its $\Z^d$-periodic counterpart. Testing against functions in $\mathcal S(\T^d)$, one can see that $\hat T'$ of $T'$ is supported on $\Z^d$ and coincides with  $\hat T$. Therefore, we identify both $T'=T$ and $ \hat T' = \hat T$. As a consequence, when applying Fourier multipliers to $T$ it does not matter whether we are using the $ \mathcal D'(\T^d)$ or the  $\Z^d$-periodic $\mathcal S ' (\R^d)$ representation of $T$.



We fix $\chi \in \mathcal{S} (\R^d)$ a non-zero function satisfying $ 0 \le \chi \le 1$ and consider the family of rescaled functions $\chi_\varepsilon : x \mapsto \varepsilon^d\chi(\varepsilon x)$. In order to make constants depending on $\chi$ universal constants, we make the explicit choice $\chi(x) = e^{-|x|^2}$. For any functions $f,g: \R_+ \rightarrow \R_+$, we denote  $f(\varepsilon) \overset{\varepsilon \to 0}{\sim} g(\varepsilon)$ and $f(\varepsilon) \overset{\varepsilon \to 0}{=} o(g(\varepsilon))$ when $\frac {f(\varepsilon)}{g(\varepsilon)}$ converges respectively to $0$ and $1$.

\begin{Prop} \label{prop:independance}
	Let $p \ge 1$ and $u \in L^p(\T^d)$. Then 
	$$ \norm{\chi_\varepsilon u }_{L^p(\R^d)}  \overset{\varepsilon \to 0}{\sim} \varepsilon^{(1-\frac 1 p)d} \norm{\chi}_{L^p(\R^d)} \norm{u }_{L^p(\T^d)}.$$
\end{Prop}
\begin{proof}
	We split $\R^d$ in cubes: $\R^d = \bigcup_{k \in \mathbb \Z^d} k+ C$ with $C = [0,1)^d$. Since $u$ is periodic, $\int_{k+C} u^p = \int_{\T^d} u^p$ for all $k \in \mathbb \Z^d$. Therefore,
	$$
	\int_{\R^d } |\chi_\varepsilon u|^p  = \sum_{k \in \mathbb \Z^d} \chi_\varepsilon(k)^p \int_{\T^d } |u|^p + \sum_{k \in \mathbb \Z^d}  \int_{k + C} (\chi_\varepsilon^p -\chi_\varepsilon(k)^p  ) |u|^p .
	$$
	On one hand, by convergence of Riemann sums,
	$$ \sum_{k \in \mathbb \Z^d} \chi_\varepsilon(k)^p  = \varepsilon^{(p-1)d} \varepsilon^d \sum_{k \in \mathbb \Z^d} \chi(\varepsilon k)^p \overset{\varepsilon \to 0}{\sim} \varepsilon^{(p-1)d} \int_{\R^d} \chi^p.$$
	On the other hand, for all $x \in k + C$, since $0 \le \chi \le 1$,
	$$|\chi_\varepsilon(x)^p -\chi_\varepsilon(k)^p  | \le p\varepsilon^{pd+1}  \norm{\nabla \chi}_{L^\infty(\varepsilon(k + C ) ) } \lesssim \frac{\varepsilon^{(p-1)d +\frac 1 2}}{1 +|k|^{d+ \frac 1 2}}  \norm{( 1 + |\cdot|^{d + \frac 1 2})\nabla \chi}_{L^\infty(\R^d) } . $$
	Thus, 
	$$ \sum_{k \in \mathbb \Z^d}  \int_{k + C} |\chi_\varepsilon^p -\chi_\varepsilon(k)^p  | |u|^p \lesssim  \varepsilon^{(p-1)d +\frac 1 2} \left(\int_{\T^d } |u|^p \right) \sum_{k \in \mathbb \Z^d} \frac {1}{1 +|k|^{d+ \frac 1 2}} \overset{\varepsilon \to 0}{=} o(\varepsilon^{(p-1)d}), $$
	hence the result.
\end{proof}

\begin{Prop} \label{prop:comutGrossier}
	Let $p \ge 2$ and $L$ be a Fourier multiplier such that $\hat L$ is continuous at each point of $\Z^d$ and bounded by a polynomial.
	Then for all  $u \in \mathcal D (\T^d)$,
	$$ \norm{ [L, \chi_\varepsilon ] u}_{L^p(\R^d)}  \overset{\varepsilon \to 0}{=} o(\varepsilon^{(1-\frac 1 p)d} ). $$ 
\end{Prop}
\begin{proof}
	First, we note that since $\hat L$ is polynomially bounded, $L u \in \mathcal D ( \T^d) $. 
	We have $\chi_\varepsilon u, \chi_\varepsilon L u \in \mathcal S (\R^d)$. Therefore, letting $g: \R^d \rightarrow \mathbb C$ be the Fourier transform of $[L, \chi_\varepsilon ] u $, we have $g \in L^1_{loc}(\R^d)$ and 
	$$ g(\xi) =  \sum_{k \in \Z^d}  \hat L(\xi) \hat \chi_\varepsilon(k - \xi) \hat u(k) -   \hat \chi_\varepsilon(k - \xi)\hat L(k) \hat u(k).$$
	Splitting $ |\hat u(k)| =  |\hat u(k)|^{\frac 1 p}|\hat u(k)|^{\frac 1 {p'}}$ where $p'$ is the conjuguate exponent of $p$ and using Hölder's inequality together with the fact that $ \chi_\varepsilon(\eta) =  \chi ( \frac \eta \varepsilon)$, one has
	$$ \norm{g}_{L^{p'}(\R^d)}^{p'} \le  \norm{\hat u }_{\ell^1(\Z^d)}^{p'-1}  \sum_{k \in \Z^d} \int_{\R^d} \left|(\hat L(\xi) -  \hat L(k) )  \hat \chi\left(\frac{k - \xi} \varepsilon\right) \right|^{p'} |\hat u(k)| \, \dd \xi ,$$
	and using the change of variables $\eta = \frac{k - \xi} \varepsilon$,
	$$ \norm{g}_{L^{p'}(\R^d)}^{p'} \le \varepsilon^d \norm{\hat u }_{\ell^1(\Z^d)}^{p'-1}  \sum_{k \in \Z^d} \int_{\R^d} \left|(\hat L(k - \varepsilon \eta) -  \hat L(k) )  \hat \chi\left(\eta \right) \right|^{p'} |\hat u(k)| \, \dd \eta .$$
	The continuity of $\hat L$ implies that the integrand converges to $0$ almost everywhere and the polynomial upper bound on $\hat L$ together with the decay of $\hat \chi$ and $\hat u$ ensure that the dominated convergence theorem applies and that the integral converges to $0$. Therefore,
	$ \norm{g}_{L^{p'}(\R^d)} = o( \varepsilon^{\frac {p'} d })$
	Since $p \ge 2$, the Hausdorff-Young inequality $\norm{[L, \chi ] u}_{L^{p}(\R^d)} \lesssim \norm{g}_{L^{p'}(\R^d)} $ concludes the proof.
\end{proof}

\begin{Cor} \label{cor:approxTransfer}
	With the assumptions of Proposition \ref{prop:comutGrossier}, for all $u,v \in \mathcal D (\T^d)$ the two following asymptotic equivalences hold
	$$ \norm{L (\chi_\varepsilon u)}_{L^p(\R^d)} \overset{\varepsilon \to 0}{\sim}\varepsilon^{(1-\frac 1 p)d} \norm{\chi}_{L^p(\R^d)}  \norm{L  u}_{L^p(\T^d)},$$
	$$ \norm{[L ,\chi_\varepsilon u]\chi_\varepsilon v }_{L^p(\R^d)}\overset{\varepsilon \to 0}{\sim}   \varepsilon^{(2-\frac 1 p)d} \norm{\chi^2}_{L^p(\R^d)} \norm{ [L,u]  v}_{L^p(\T^d)}   .$$ 
\end{Cor}
\begin{proof}
	The right-hand side of the first asymptotic equivalence corresponds to $\norm{ \chi_\varepsilon L u}_{L^p(\R^d)}$, as explained in Proposition \ref{prop:independance}. The remainder is the commutator estimated in Proposition \ref{prop:comutGrossier},
	$$| \norm{L (\chi_\varepsilon u)}_{L^p(\R^d)} -  \norm{ \chi_\varepsilon L u}_{L^p(\R^d)} | \le \norm{ [L, \chi_\varepsilon ] u}_{L^p(\R^d)} = o(\varepsilon^{(1-\frac 1 p)d}),$$
	which is sufficient to conclude the first asymptotic equivalence.
	
	For the second one, we use the following decomposition: 
	\begin{align*}
		[L ,\chi_\varepsilon u]\chi_\varepsilon v &=   L( \chi_\varepsilon u \chi_\varepsilon v) - \chi_\varepsilon u L(  \chi_\varepsilon v) \\
		&= [L ,\chi_\varepsilon^2](uv) + \chi_\varepsilon^2 L(uv) -  \chi_\varepsilon u [L, \chi_\varepsilon] v -  \chi_\varepsilon^2 u L  v \\
		&=  [L ,\chi_\varepsilon^2](uv)  -  \chi_\varepsilon u [L, \chi_\varepsilon] v +  \chi_\varepsilon^2 [L,u]  v.
	\end{align*}

	Setting $\tilde \chi = \chi^2 \in \mathcal{S}(\R^d)$ and $ \tilde \chi _\varepsilon(x) = \varepsilon^d \tilde \chi( \varepsilon x)$, we have $ \chi_\varepsilon^2 = \varepsilon^d  \tilde \chi_\varepsilon$ and applying Proposition \ref{prop:comutGrossier}
	$$\norm{[L ,\chi_\varepsilon^2](uv)}_{L^p(\R^d)} = \varepsilon^d \norm{[L ,\tilde \chi_\varepsilon](uv)}_{L^p(\R^d)}= o(\varepsilon^{(2-\frac 1 p)d} ) . $$
	The next term is bounded using Proposition \ref{prop:comutGrossier} after the Hölder inequality:
	$$\norm{ \chi_\varepsilon u [L, \chi_\varepsilon] v }_{L^p(\R^d)} \le\norm{ \chi_\varepsilon u}_{L^\infty (\R^d)} \norm{ [L, \chi_\varepsilon] v }_{L^p(\R^d)} \le  \varepsilon^d \norm{ \chi }_{L^\infty (\R^d)} \norm{ u}_{L^\infty (\T^d)} o( \varepsilon^{(1-\frac 1 p)d})= o( \varepsilon^{(2-\frac 1 p)d}).$$
	The last term is the dominant term since by Proposition \ref{prop:independance},
	$$ \norm{ \chi_\varepsilon^2 [L,u]  v }_{L^p(\R^d)} =  \varepsilon^d \norm{ \tilde \chi_\varepsilon [L,u]  v }_{L^p(\R^d)}\overset{\varepsilon \to 0}{\sim}  \varepsilon^{(2-\frac 1 p)d} \norm{\chi^2}_{L^p(\R^d)} \norm{ [L,u]  v}_{L^p(\T^d)}.$$
	
\end{proof}

\begin{proof}[Proof of Lemma \ref{lem:transference}]
	For $(i)$, let $u \in \mathcal D(\T^d)$. We apply the estimate to $\chi_\varepsilon u \in  \mathcal D(\T^d)$:
	$$ \norm{L_0(\chi_\varepsilon u) }_{L^p(\R^d)}  \lesssim  \norm{L_{1} (\chi_\varepsilon u) }_{L^{p_1}(\R^d)}^{\theta}   \norm{L_{2} (\chi_\varepsilon u) }_{L^{ p_2}(\R^d)}^{1 - \theta} .$$ 
	Then Corollary \ref{cor:approxTransfer} yields 
	\begin{multline*}
		\varepsilon^{(1-\frac 1 p)d} \norm{L_0 u }_{L^p(\T^d)} 
		\lesssim  \varepsilon^{\theta (1-\frac 1 {p_1})d + (1 - \theta) (1-\frac 1 {p_2})d }\norm{L_{1} u }_{L^{p_1}(\T^d)}^{\theta}   \norm{L_{2} (\chi_\varepsilon u) }_{L^{ p_2}(\T^d)}^{1 - \theta}\\
		+ o( \varepsilon^{(1-\frac 1 p)d} + \varepsilon^{\theta (1-\frac 1 {p_1})d + (1 - \theta) (1-\frac 1 {p_2})d }) .
	\end{multline*} 
	Both powers on $\varepsilon$ are the same by assumption on $p, p_1, p_2, \theta$ so dividing by $ \varepsilon^{(1-\frac 1 p)d}$ and sending $\varepsilon$ to $0$ we get the estimate on $\T^d$.
	
	The proof of $(ii)$ is similar; once the original estimate on $\R^d$ has been applied to $\chi_\varepsilon u$ and $\chi_\varepsilon v$, one remarks that all the terms scale with the power of $\varepsilon$, yielding the estimate on $\T^d$ at the limit $\varepsilon \to 0$.

\end{proof}

As a conclusion of this subsection, we transfer the Gagliardo-Nirenberg inequality. 
\begin{Prop}  \label{prop:GN}
	Let $p, p_0, p_1 \in [2, \infty]$, $s, s_0, s_1 \ge 0$ and $\theta \in [0,1]$ such that $\frac 1 p = \frac \theta {p_1} +  \frac {1-\theta} {p_0} $ and $s= \theta p_1 + (1- \theta) p_0 $. For all $u \in \mathcal D (\T^d)$,
	$$ \norm{J^s u }_{L^p(\T^d)}  \lesssim  \norm{J^{s_1} u }_{L^{p_1}(\T^d)}^{\theta}   \norm{J^{s_0} u }_{L^{ p_0}(\T^d)}^{1 - \theta} .$$ 
\end{Prop}
\begin{proof}
	Lemma 3.1 of \cite{brezis_gagliardo-nirenberg_2001} contains the version of this result on $\R^d$. Since the $J^s$ are Fourier multipliers satisfying the assumptions of Lemma \ref{lem:transference}, the estimate transfers to $\T^d$.
\end{proof}

\subsection{Commutator estimates}

We derive all of our commutator estimates from the following one, which is a transfer from $\R^d$ to $\T^d$ of Theorem 1.9 in \cite{li_kato-ponce_2016}.
\begin{Prop} \label{prop:liComut}
	Let $s\ge 0$ and $r, p,\bar p, q, \bar q \in [2, \infty]$ such that $\frac 1 p + \frac 1 {\bar p} \le \frac 1 r$ and $\frac 1 q + \frac 1 {\bar q}  \le \frac 1 r$. Then for any $u,v \in \mathcal D(\T^d)$, the following estimate holds:
	$$\norm{[J^{s}, u]v}_{L^{r}(\T^d)} \lesssim \norm{J^{s-1} \nabla u}_{L^{\bar p}(\T^d)} \norm{v}_{L^p(\T^d)} + \mathbf 1 _{s > 1} \norm{ \nabla u}_{L^{\bar q}(\T^d)} \norm{J^{s-2}\nabla v}_{L^{q}(\T^d)},$$
	where $ \mathbf 1 _{s > 1}$ indicates that the second term is present only if $s > 1$.
\end{Prop}
\begin{Rmk}
	Although the lemma is stated for smooth functions, by a density argument it can be extended to any functions $u,v : \T^d \rightarrow \R$ such that the above norms make sense.
\end{Rmk}
\begin{proof}
	Note that since $L^a(\T^d) \hookrightarrow L^b(\T^d)$ when $a \ge b$, it is sufficient to prove the estimate in the case $\frac 1 p + \frac 1 {\bar p}  =\frac 1 q + \frac 1 {\bar q}  = \frac 1 r$.
	By Proposition \ref{prop:liComutR}, this estimate holds in this case when the domain is replaced by $\R^d $ and $u,v \in \mathcal S(\R^d )$. Since $J^s$ and $\nabla$ are polynomial and continuous Fourier multipliers, the transference principle in Lemma \ref{lem:transference} applies and the estimate holds on $\T^d$.
\end{proof}

\begin{Cor} \label{cor:kato}
	Let $s> d/ 2$.  Then for all $u\in H^{s+1}(\T^d)$ and $v\in H^{s}(\T^d)$,
	$$ \norm{[J^{s+1} , u]v}_{L^2} \lesssim  \norm{\nabla u}_{\hsp}\norm{v}_{\hs}.$$ 
\end{Cor}
\begin{proof}
	Since by Sobolev embedding $\norm{\cdot}_{L^\infty(\T^d)} \lesssim \norm{\cdot}_{H^s(\T^d)} =  \norm{J^s \cdot}_{L^2(\T^d)}$, the wanted estimate is a consequence of
	$$\norm{[J^{s+1}, u]v}_{L^{2}(\T^d)} \lesssim \norm{J^{s} \nabla u}_{L^{2}(\T^d)} \norm{v}_{L^\infty(\T^d)} + \norm{ \nabla u}_{L^{\infty}(\T^d)} \norm{J^{s-1}\nabla v}_{L^2(\T^d)}.$$
	This last estimate holds by Proposition \ref{prop:liComut}.
\end{proof}

\begin{Lemma} \label{lem:comutLp}
	Let $s> \frac d 2$ and $\varepsilon \ge 0$. Then there exists $p \in [2, \infty]$ such that for all $\delta >0$ there exists an increasing function $P_\delta: \R^2 \rightarrow \R_+$ such that for all $u\in H^{s+\varepsilon}(\T^d)$ and $v\in H^{s}(\T^d)$,
	$$ \norm{[J^{s}, u]v}_{L^{2}} \lesssim P_\delta(\norm{ u}_{L^p}, \norm{ v}_{L^p}) (1+\norm{u}_{H^{s +\varepsilon}}) + \delta \norm{ v}_{H^{s}}.$$
	Moreover, if $\varepsilon > 0$ then $p$ can be chosen finite and $P_\delta$ is polynomial.
\end{Lemma}
\begin{proof}
	Let $p > \max(d/ \varepsilon, 2)$ ($p= \infty$ if $\varepsilon=0$), $\bar p$ such that $ \frac 1 p + \frac 1 {\bar p}  =  \frac 1 2$ and $ \bar q, q$ to be set later such that $ \frac 1 q + \frac 1 {\bar q}  \le  \frac 1 2$. By Proposition \ref{prop:liComut},
	\begin{equation} \label{eq:liComut}
		\norm{[J^{s}, u]v}_{L^{2}} \lesssim \norm{J^{s-1} \nabla u}_{L^{\bar p}} \norm{v}_{L^p} + \norm{ \nabla u}_{L^{\bar q}} \norm{J^{s-2}\nabla v}_{L^{q}}
	\end{equation}
	Since $p > \max(d/ \varepsilon, 2)$, we have $\frac 1 2 - \frac \varepsilon d \le \frac 1 {\bar p}$ and $\bar p < \infty$ so the Sobolev embedding $H^\varepsilon(\T^d) \hookrightarrow L^{\bar p}(\T^d)$ holds, as well as the following bound on the first term of the right-hand side of \eqref{eq:liComut}:
	$$  \norm{J^{s-1} \nabla u}_{L^{\bar p}} \norm{v}_{L^p} \lesssim \norm{J^{s -1} \nabla u}_{\varepsilon} \norm{v}_{L^p} \lesssim \norm{u}_{H^{s +\varepsilon}} \norm{v}_{L^p}. $$
	For the other term, we assume that at the beginning of the proof we chose the $\bar q, q$ given by Lemma \ref{lem:calculGN} and $p$ large enough to have
	$$	\norm{ \nabla u}_{L^{\bar q}}  \norm{J^{s-2}\nabla v}_{L^{q}}\le P_\delta(  \norm{ u  }_{L^{p}},\norm{ v  }_{L^{p}} )(1 +   \norm{ u  }_{H^{s+\varepsilon}}) + \delta  \norm{ v  }_{H^{s}} ,  $$ 
	where $P_\delta$ is an increasing function (polynomial if $\varepsilon>0$) depending only on $\delta$.
	Plugging this and the previous estimate and in \eqref{eq:liComut} we get the result.
\end{proof}

\subsection{Composition in Sobolev spaces} \label{sec:compo} 

We need a variant of the following result, which is contained in a more general form in \cite{bahouri_fourier_2011}, Theorem 2.89 and Corollary 2.91:

\begin{Prop}\label{prop:bahouri} Let $s> d/ 2$ and $F: \R^n \rightarrow \R^m $ be a smooth function. Then for all functions $U,V \in H^s(\T^d)$, $F \circ U, F \circ v \in H ^s(\T^d)^m$ and
	$$ \norm{F \circ U - F \circ V }_{\hs} \lesssim_{F, \norm{U}_{L^\infty}, \norm{V}_{L^\infty}} \norm{U - V}_{\hs} . $$
	In particular, in the case $V=0$,
	$$ \norm{F \circ U }_{\hs} \lesssim_{F, \norm{U}_{L^\infty}}1 + \norm{U}_{\hs} . $$ 
	
\end{Prop}
In the case where the $\lceil s \rceil$ first derivatives of $F$ are bounded by a polynomial, we will relax the exigence of an $L^\infty$ norm behind $\lesssim$ and only ask for an $L^p$ norm with $p<\infty$ at the price of needing more regularity on $U$. Before stating our result, let us define the Bessel potential spaces $H ^{s,p}(\T ^d)$. For $s \in \R$ and $p \ge 1$, a distribution $u \in \mathcal D'(\T ^d)$ lies in $H ^{s,p}(\T ^d)$ if and only if $J^s u \in L^{p}(\T ^d)$ and in this case we set 
$$ \norm{u }_{H ^{s,p}} = \norm{J^s u }_{L ^p}.$$

\begin{Prop} \label{prop:wekaerComp}
	Let $n\in \mathbb N^*$, $s \ge 0$, $r \in [2, \infty)$, $\varepsilon>0$ , and $f\in  \mathcal C^{\infty}(\R^n; \R) $ such that there exists constants $C,\alpha>0$ satisfying $|\nabla^{\lceil s \rceil}f(X)| \le C(1 + |X|^\alpha)$. Then, there exists $p<\infty$ and a polynomial function $P$ such that for all $U \in H^{s + \varepsilon,r}(\T^d) $,
	$$\norm{f(U)}_{H ^{s,r}}  \le P(\norm U _{L ^p} ) (1 + \norm U _ {H ^{s+ \varepsilon, r }} ). $$
\end{Prop}

The following piece of calculation will be used multiple times:
\begin{Lemma} \label{lem:calculGN}
	Let $s , t, \varepsilon \ge 0$ and $r \in [2, \infty) $. Then there exists $q, \bar q, p \in [2, \infty]$ such that $\frac 1 {\bar q} + \frac 1 q \le \frac 1 r$ and for all $\delta >0$ there exists a polynomial function $P_\delta: \R \rightarrow \R_+$ such that for all $u \in H^{s+t+\varepsilon,r }(\T ^d) \cap L^{p}(\T ^d) $ and $v \in H^{s+t, r}(\T ^d) \cap L^{p}(\T ^d)$,
	$$ \norm{ u  }_{H^{s, \bar q}} \norm{ v  }_{H^{t,q}}\le P_\delta(  \norm{ u  }_{L^{p}},\norm{ v  }_{L^{p}} )(1 +   \norm{ u  }_{H^{s+t+\varepsilon, r}}) + \delta  \norm{ v  }_{H^{s+t, r }} .  $$ 
	If $\varepsilon > 0$, $p$ can be chosen finite and $P_\delta$ is polynomial. In particular, for $r=2$,
	$$	\norm{ \nabla u}_{L^{\bar q}}  \norm{J^{s-2}\nabla v}_{L^{q}}\le P_\delta(  \norm{ u  }_{L^{p}},\norm{ v  }_{L^{p}} )(1 +   \norm{ u  }_{H^{s+\varepsilon}}) + \delta  \norm{ v  }_{H^{s}} ,  $$ 
	with $ p < \infty$ and $P_\delta$ polynomial if $\varepsilon > 0$.
\end{Lemma}

\begin{proof}
	If $s = t =0$, the result holds for the choice $q = \bar q = p = 2r$. Let $p \in [2, \infty]$ to be chosen later and set 
	$$ \bar \theta = \frac {s}{s+t+ \varepsilon}, \quad \theta = \frac t {s+t}, \quad \frac 1 {\bar q} = \frac {1 - \bar \theta} {p} + \frac {\bar \theta} {r}, \quad \frac 1 { q} = \frac {1- \theta} {p} + \frac { \theta} {r}.$$
	Observe that $\frac {\bar \theta + \theta} r \le \frac 1 r$ and that the inequality is strict if $\varepsilon > 0$ so there exists a choice of $p  \in [2, \infty]$ such that $ \frac 1 {\bar q} + \frac 1 {q} \le \frac 1 r$ and if $\varepsilon > 0$, $p$ can be chosen finite. We fix such a $p  \in [2, \infty]$ and apply the Gagliardo-Nirenberg interpolation inequality (see Proposition \ref{prop:GN}):
	$$\norm{ J ^s  u}_{L^{\bar q}} \lesssim \norm{ u}_{L^p}^{1- \bar \theta}\norm{ u}_{H^{s+t+\varepsilon,r}}^{\bar \theta} , \qquad  \norm{J^{t} v}_{L^{q}} \lesssim \norm{ v}_{L^p}^{1 - \theta}\norm{ v}_{H^{s+t,r}}^{\theta}. $$
	Therefore, using Young's inequality, for all $\delta >0$,
	$$
	\norm{ J ^s  u}_{L^{\bar q}}\norm{J^{t} v}_{L^{q}} 
	\lesssim  C_\delta \left(\norm{ u}_{L^p}^{1- \bar \theta}\norm{ u}_{H^{s+t+\varepsilon,r}}^{\bar \theta} \norm{ v}_{L^p}^{1 - \theta} \right)^{\frac 1 {1 - \theta}} +\delta\norm{ v}_{H^{s+t,r}},
	$$
	where $C_\delta$ is a constant depending only on $\delta$ and $\theta$.
	Computing $ \frac {\bar \theta} {1 - \theta} = \frac {s+t} {s+t+\varepsilon} \le 1$, we deduce that $\norm{ u}_{H^{s+t+\varepsilon,r}}^{  {\bar \theta} / {1 - \theta}} \le 1 + \norm{ u}_{H^{s+t+\varepsilon,r}} $. Plugging this back in the above inequality, we have the result.
\end{proof}

\begin{proof}[Proof of Proposition \ref{prop:wekaerComp}]
	By a density argument, it is sufficient to prove the proposition in a smooth setting. From now on, we only work with smooth functions. The proof is by induction.
	
	\emph{Case $0 \le s < 1$.} If $s=0$, the result comes directly from the fact that $f$ is polynomially bounded. If $0 < s < 1$, we prefer to work with the integral kernel representation of the Sobolev-Slobodeckij norms $\norm{ \cdot}_{W^{s,r}} $, so we let $\delta \in (0, \varepsilon)$ such that $s + \delta < 1$ and rely on the bounds
	\begin{equation} \label{eq:besselSob}
		\norm{ \cdot}_{H^{s,r}} \lesssim \norm{ \cdot}_{W^{s + \delta ,r}}  \lesssim  \norm{ \cdot}_{H^{s + \varepsilon ,r}} .
	\end{equation} 
	These inequalities can be seen from the representations of $H^{s,r}= F^{s}_{r,2}$ and $W^{s,r} = F^{s}_{r,r}$ as Triebel-Lizorkin spaces and the elementary embedding $  F^{s_2}_{r,q_2} \hookrightarrow F^{s_1}_{r,q_1}$ for all $r \in (1, \infty)$, $q_1, q_2 \ge 1$ and $s_2 > s_1$. The theory of  Triebel-Lizorkin spaces is detailed in \cite{triebel_theory_1983} and an account of the fractional Sobolev spaces $W^{s,r}$ is given in \cite{nezza_hitchhikers_2011}.
	We now apply the first inequality in \eqref{eq:besselSob} to $f(U)$:
	$$ 	\norm{ f(U)}_{H^{s,r}}^r  \le  \norm{ f(U)}_{W^{s + \delta ,r}}^r =  \norm{ f(U)}_{L^{r}}^r +  \iint_{\T^{2d}} \frac {|f(U(x)) - f(U(y))| ^r }{|x - y |^{d + (s+\delta)r}} \,  \dd x \dd y .$$
	$f$ is bounded by a polynomial so $\norm{ f(U)}_{L^{r}}^r$ is already bounded by a polynomial in some $\norm{U}_{L^{p}}$ norm. The assumption on $\nabla f$ also implies that for all $Z \in [f(U(x)), f(U(y))]$, $\nabla f (Z) \lesssim 1 + |U(x)|^\alpha + |U(y)|^\alpha$. Therefore, applying the mean value theorem,
	$$ 	\norm{ f(U)}_{H^{s,r}}^r  \lesssim   P(\norm{U}_{L^{p}} ) +  \iint_{\T^{2d}} (1 + |U(x)|^\alpha + |U(y)|^\alpha)^r \frac {|U(x) - U(y)| ^r }{|x - y |^{d + (s+\delta)r}} \,  \dd x \dd y ,$$
	where $P$ is a polynomial. Let $\tilde r > r$ to be fixed later and $\tilde r'$ its conjugate exponent. By Hölder's inequality, 
	$$ 	\norm{ f(U)}_{H^{s,r}}  \lesssim   P(\norm{U}_{L^{p}} ) + \left(  \iint_{\T^{2d}} (1 + |U(x)|^\alpha + |U(y)|^\alpha)^{\tilde r} \,  \dd x \dd y  \right)^{\frac 1 {\tilde r'}} \hspace{-4pt} \left( \iint_{\T^{2d}} \frac {|U(x) - U(y)| ^{\tilde r} }{|x - y |^{d\tilde r / r  + (s+\delta)\tilde r}} \,  \dd x \dd y  \right)^{\frac 1 {\tilde r}}$$
	Setting $\tilde s =d(\frac 1 r - \frac 1 {\tilde r}) + s + \delta $ and computing $d\tilde r / r  + (s+\delta)\tilde r = d + \tilde s\tilde r$, this rewrites for some $p\ge 1$
	$$ 	\norm{ f(U)}_{H^{s,r}}  \lesssim   P(\norm{U}_{L^{p}} )( 1 +	\norm{ U}_{W^{\tilde s, \tilde r}} ).$$
	Notice that as $\tilde r \rightarrow r$, $\tilde s \rightarrow s + \delta$ so $\tilde r$ can be chosen close enough to $r$ to have the existence of some $\tilde \delta < \varepsilon$ such that the Sobolev embedding $W^{s+ \tilde \delta, r }(\T ^d ) \hookrightarrow W^{\tilde s, \tilde r }(\T ^d )$ holds. The second member of \eqref{eq:besselSob} then tells us that $\norm{ \cdot}_{W^{s + \tilde \delta ,r}}  \lesssim  \norm{ \cdot}_{H^{s + \varepsilon ,r}}$ and we get
	$$ 	\norm{ f(U)}_{H^{s,r}}  \lesssim   P(\norm{U}_{L^{p}} )( 1 +	\norm{ U}_{H^{s + \varepsilon ,r}} ),$$
	which is the result in the case $s \in (0,1)$.

	\emph{Case $s \ge 1$.} Assume that the theorem holds up to the regularity exponent $s-1$. Then, we have
	\begin{equation} \label{eq:6101}
		\norm{ f(U)}_{H^{s,r}} \lesssim  \norm{  f(U)}_{H^{s-1,r}} + \norm{ \nabla( f(U))}_{H^{s-1,r}}.
	\end{equation} 	
	We bound the first term in the right-hand side using the induction hypothesis. We now focus on the second one. Using the chain rule,
	\begin{equation} \label{eq:6102}
		J^{s-1} \left(\nabla( f(U))\right) =  J^{s-1} \left(\nabla f (U) \nabla U \right) =\nabla f (U) J^{s-1} \nabla U + [J^{s-1} , \nabla f(U)] \nabla U  .
	\end{equation}
	For the first term in the right-hand side, we let $\tilde r > r$ such that the Sobolev embedding $H^{s+ \varepsilon,r} (\T^d) \hookrightarrow H^{s, \tilde r}(\T^d)$ holds and apply Hölder's inequality and the polynomial bound on $\nabla f$:
	\begin{equation} \label{eq:6103}
		\norm{\nabla f (U) J^{s-1} \nabla U}_{L^r} \lesssim \norm{\nabla f (U) }_{L^{r\tilde r / (\tilde r - r)}} \norm{ J^{s-1} \nabla U}_{L^{\tilde r}} \lesssim P(\norm{U }_{L^{p}})  \norm{ U}_{H^{s + \varepsilon, r }}, 
	\end{equation} 
	where $P$ is a polynomial. 
	For the commutator part, let $\bar q_1, q_1, \bar q_2, q_2$  to be chosen later and satisfying the assumptions of Proposition \ref{prop:liComut}. Then,
	\begin{align*}
		\norm{[J^{s-1},  \nabla f(U)]\nabla U}_{L^{r}} &\lesssim \norm{J^{s-2} \nabla ( \nabla f(U) )}_{L^{\bar {q_1}}} \norm{\nabla U}_{L^{q_1}} + \mathbf 1 _{s > 2} \norm{ \nabla ( \nabla f(U)) }_{L^{\bar q_2}} \norm{J^{s-3}\nabla^2 U}_{L^{q_2}(\T^d)} \\
		&\lesssim \norm{ \nabla f(U) }_{H^{s-1, \bar q_1}} \norm{ U}_{H^{1, q_1}} + \mathbf 1 _{s > 2} \norm{ \nabla f(U) }_{H^{1,\bar q_2}} \norm{ U}_{H^{s-1,q_2}}.
	\end{align*} 
	We now use the induction hypothesis on the composition by $\nabla f$ (it applies since $ \nabla^{\lceil s \rceil -1} \nabla f$ has a polynomial bound) and get that for some $p \in (1, \infty)$
	\begin{multline*}
		\norm{[J^{s-1},  \nabla f(U)]\nabla U}_{L^{r}} \\
		\lesssim P(\norm U _{L ^p} ) \left( (1 + \norm{ U }_{H^{s-1 + \varepsilon/2, \bar q_1}}) \norm{ U}_{H^{1, q_1}} + \mathbf 1 _{s > 2} (1 + \norm{ U }_{H^{1 + \varepsilon/2,\bar q_2}} ) \norm{ U}_{H^{s-1,q_2}} \right), 
	\end{multline*} 
	where $P$ is a polynomial.
	We then choose the $\bar q_i, q_i$ given by Lemma \ref{lem:calculGN} in order to have (up to taking $p$ and $P$ large enough)
	$$\norm{ U }_{H^{s-1 + \varepsilon/2, \bar q_1}} \! \norm{ U}_{H^{1, q_1}} , \hspace{4mm}  \norm{ U }_{H^{1 + \varepsilon/2,\bar q_2}} \!  \norm{ U}_{H^{s-1,q_2}} \lesssim   P(\norm U _{L ^p} ) ( 1 + \norm{ U}_{H^{s+ \varepsilon,r}}).$$
	Seeing the constant $1$ as some $H^{t,\bar q_i}$ norm of the constant function $1$, one also has
	$$ \norm{ U}_{H^{1, q_1}} , \hspace{4mm}   \norm{ U}_{H^{s-1,q_2}} \lesssim  P(\norm U _{L ^p} ) ( 1 + \norm{ U}_{H^{s+ \varepsilon,r}}).$$
	We proved 
	\begin{equation*}
		\norm{[J^{s-1},  \nabla f(U)]\nabla U}_{L^{r}} 
		\lesssim P(\norm U _{L ^p} ) ( 1 + \norm{ U}_{H^{s+ \varepsilon,r}}),
	\end{equation*} 
	and plugging this and \eqref{eq:6103} back in \eqref{eq:6102} we get
	$$\norm{ \nabla( f(U))}_{H^{s-1,r}} =  \norm{J^{s-1} (\nabla( f(U)))}_{L^r} \lesssim P(\norm U _{L ^p} ) ( 1 + \norm{ U}_{H^{s+ \varepsilon,r}}) .$$
	Finally, plugging this estimate in \eqref{eq:6101} gives the final result.
\end{proof}

\printbibliography[
]

\end{document}